\documentclass[11pt, reqno]{amsart}
\usepackage[a4paper,vmargin={3.5cm,3.5cm},hmargin={2.5cm,2.5cm},bottom=30mm]{geometry}
\usepackage{amsfonts}
\usepackage{amsthm,amsmath}
\usepackage{amsbsy}
\usepackage{bm}
\usepackage{amssymb} 
\usepackage{verbatim}
 \usepackage{color}

\usepackage[mathscr]{eucal}
\usepackage{times}

\usepackage{units}
\usepackage{enumerate}
\usepackage{bbm}
\usepackage{exscale}
\usepackage{hyperref}
\hypersetup{ colorlinks=true, linkcolor=blue, citecolor=blue,
  filecolor=blue, urlcolor=blue}
  \numberwithin{equation}{section} 
  \date{}

\usepackage[dvipsnames]{xcolor}
\theoremstyle{plain} 
    \newtheorem{theorem}{Theorem}
    \newtheorem{lemma}[theorem]{Lemma}
    \newtheorem{proposition}[theorem]{Proposition}
    
    \newtheorem{claim}[theorem]{Claim}

\theoremstyle{definition} 
    \newtheorem{definition}{Definition}

    \newtheorem{remark}[definition]{Remark}

\usepackage{appendix}
\renewcommand\appendix{\par
\setcounter{section}{0}%
\setcounter{subsection}{0}%
\setcounter{table}{0}
\setcounter{figure}{0}
\gdef\thetable{\Alph{table}}
\gdef\thefigure{\Alph{figure}}
\section*{Appendix}
\gdef\thesection{\Alph{section}}
\setcounter{section}{0}}

\DeclareMathOperator{\T}{\textbf{T}}
\DeclareMathOperator{\R}{\mathbb{R}}

\DeclareMathOperator{\Z}{\mathbb{Z}}
\DeclareMathOperator{\N}{\mathbb{N}}
\usepackage{mathtools}

\DeclareMathOperator{\De}{d}
\DeclareMathOperator{\one}{\mathbbm{1}} 
\newcommand{\E}{\mathsf{E}}

\newcommand{\var}{\mathsf{Var}}

\newcommand{\prob}{\mathsf{P}}
\renewcommand{\T}{\mathbb{T}}

\DeclareMathOperator{\e}{e}
\renewcommand{\O}[1]{\mathrm{O}\left(#1\right)} 
\renewcommand{\o}[1]{\mathrm{o}\left(#1\right)} 
\newcommand{\la}{\left\langle}
\newcommand{\ra}{\right\rangle}

\newcommand{\eps}{\epsilon}
\newcommand{\cR}{\mathcal R}

\newcommand{\eq}[1]{\begin{equation#1}}
\newcommand{\eeq}[1]{\end{equation#1}}

\renewcommand{\i}{\iota}
\usepackage[comma, authoryear]{natbib}

\begin{document}
\title[Scaling limit of odometer]{\bf \textsc{Scaling limit of the odometer in divisible sandpiles}}
\author{Alessandra Cipriani,   Rajat Subhra Hazra   and   Wioletta M. Ruszel}
\address{University of Bath}
\email{A.Cipriani@bath.ac.uk}

\address{Indian Statistical Institute, Kolkata.}
\email{rajatmaths@gmail.com}
\address{TU Delft}
\email{W.M.Ruszel@tudelft.nl}

\begin{abstract}
In a recent work \cite{LMPU} prove that the odometer function of a divisible sandpile model on a finite graph can be expressed as a shifted discrete bilaplacian Gaussian field. For the discrete torus, they suggest the possibility that the scaling limit of the odometer may be related to the continuum bilaplacian field. In this work we show that in any dimension the rescaled odometer converges to the continuum bilaplacian field on the unit torus.
\end{abstract}

\maketitle
\section{Introduction}
The concept of self-organized criticality was introduced in \cite{BTW} as a lattice model with a fairly elementary dynamics. Despite its simplicity, this model exhibits a very complex structure: the dynamics drives the system towards a stationary state which shares several properties of equilibrium systems at the critical point, e.g. power law decay of cluster sizes and of correlations of the height-variables. The model was generalised by \cite{Dha90} in the so-called Abelian sandpile model (ASP). Since then, the study of self-criticality has become popular in many fields of natural sciences, and we refer the reader to \cite{JaraiCornell} and \cite{RedigNotes} for an overview on the subject. In particular, several modifications of the ASP were introduced such as non-Abelian models, ASP on different geometries, and continuum versions like the divisible sandpile treated in \cite{LevPer,LePe10}. We are interested in the latter one which is defined as follows. By a graph $G = (V, E)$ we indicate a connected, locally finite and undirected graph with vertex set $V$ and edge set $E$. By $\mathrm{deg}(x)$ we denote the number of neighbours of $x\in V$ in $E$ and we write ``$y\sim_V x$'' when $(x,\,y)\in E$. A divisible sandpile configuration on G is a function $s : V \to\R$, where $s(x)$ indicates a mass of particles at site $x$. Note that here, unlike the ASP, $s(x)$ is a real-valued (possibly negative) number. If a vertex $x \in V$ satisfies $s(x)> 1$, it topples by keeping mass $1$ for itself and distributing the excess $s(x) -1$ uniformly among its neighbours. At each discrete time step, all unstable vertices topple simultaneously.

Given $(\sigma(x))_{x\in V}$ i.i.d.~standard Gaussians, we construct the divisible sandpile with weights $(\sigma(x))_{x\in V}$ by defining its initial configuration as 
\eq{}\label{eq:density}s(x)=1+\sigma(x)-\frac1{|V|}\sum_{y\in V} \sigma(y).\eeq{}
As in many models of statistical mechanics, one is interested in defining a notion of criticality here too. 

 Let $e^{(n)}(x)$ denote the total mass distributed by $x$ before time $n$ to any of its neighbours. If $e^{(n)}(x)\uparrow e_{V}$ where $e_{V}:\, V\to [0,\,+\infty]$, then $e_{V}$ is called the {\em odometer} of $s$. We have the following dichotomy: either $e_{V}<+\infty$ for all $x\in V$ (stabilization), or $e_{V}=+\infty$ for all $x\in V$ (explosion). It was shown in \cite{LMPU} that if $s(x)$ is assumed to be i.i.d.\ on an infinite graph which is vertex transitive, and if $\E[s(x)]>1$,  $s$ does not stabilize, while stabilization occurs for $\E[s(x)]<1$. In the critical case ($\E[s(x)]=1$) the situation is graph-dependent. For an infinite vertex transitive graph, with $\E[s(x)]=1$ and $0<\var(s(x))<+\infty$ then $s$ almost surely does not stabilize.

 For a finite connected graph, one can give quantitive estimates and representations for $e_V$. It is shown in \citet[Proposition 1.3]{LMPU} that the odometer corresponding to the density \eqref{eq:density} on a finite graph $V$ has distribution
\[(e_V(x))_{x\in V}\overset{d}=\left(\eta(x) -\min_{z\in V}\eta(z)\right)_{x\in V}\]
where $\eta$ is a ``bilaplacian" centered Gaussian field with covariance
$$\E[\eta(x)\eta(y)] =\frac{1}{\mathrm{deg}(x)\mathrm{deg}(y)}\sum_{w\in V}g(x, w)g(w, y)$$
setting
\eq{}\label{eq:g}g(x,y)= \frac1{|V|}\sum_{z\in V} g^z(x,y)\eeq{}
and $g^z(x, y)=\E\left[\sum_{m=0}^{\tau_z-1}\one_{\{S_m=y\}}\right]$ for $S=(S_m)_{m\ge 0}$ a simple random walk on $V$ starting at $x$ and $\tau_z := \inf\{m\ge 0 : S_m = z\}$. The field is called ``bilaplacian'' since a straightforward computation shows that
\[\Delta^2_g\left(\frac{1}{\mathrm{deg}(x)\mathrm{deg}(y)}\sum_{w\in V}g(x, w)g(w, y)\right)=\delta_x(y)-\frac{1}{|V|}
\]
where $\Delta_g$ denotes the graph Laplacian 
\[
\Delta_g f(x):=\sum_{y\sim_V x}f(y)-f(x),\quad f:\,V\to\R.
\]
Hence the covariance is related to the Green's function of the discrete bilaplacian (or biharmonic) operator. 

The interplay between the odometer of the sandpile and the bilaplacian becomes more evident in the observation made by Levine et al. on the odometer in $V:=\Z_n^d$, the discrete torus of side length $n>0$ in dimension $d$. They write (after the statement of Proposition 1.3):
\begin{quotation}
``We believe that if $\sigma$ is identically distributed with zero mean and finite variance, then the odometer, after a suitable shift and rescaling, converges weakly as $n\to+ \infty$ to the bilaplacian Gaussian field on $\R^d$".
\end{quotation}
Note that, although they work with Gaussian weights in the proof of Proposition~1.3, their comment comprises also the case when $\sigma$ has a more general distribution. Inspired by the above remark, we determine the scaling limit of the odometer in $d\ge 1$ for general i.i.d. weights: we show that indeed it equals $\Xi$, the continuum bilaplacian, but on the unit torus $\T^d$ (see Theorems~\ref{thm:main} and \ref{thm:3}). A heuristic for the toric limit is that the laplacian we consider is on $\Z_n^d$, which can be seen as dilation of the discrete torus $\T^d\cap(n^{-1}\Z)^d$. We highlight that $\Xi$ is not a random variable, but a {\em random distribution} living in an appropriate Sobolev space on $\T^d$. There are several ways in which one can represent such a field: a convenient one is to let $\Xi$ be a collection of centered Gaussian random variables $\left\{\la \Xi,\,u \ra:\,u\in H^{-1}(\T^d)\right\}$ with variance $\E\left[\la \Xi,\,u \ra^2\right]=\|u\|^2_{{-1}}$, where 
\[
\|u\|^2_{{-1}}:=\left(u,\Delta^{-2}u\right)_{L^2(\T^d)}
\]
and $\Delta^2$ now is the continuum bilaplacian operator.
We will give the analytical background to this definition in Subsection~\ref{subsec:review}. As a by-product of our proof, we are able to determine the kernel of the continuum bilaplacian on the torus which, to the best of the authors' knowledge, is not explicitly stated in the literature.
\paragraph{\it Related work.} Scaling limits for sandpiles have already been investigated: in the ASP literature limits for stable configurations have been studied, for example, in \cite{levine2012} and \cite{pegden:smart}. Their works are concerned with the partial differential equation that characterizes the scaling limit of the ASP in $\Z^2$. They also provide an interesting explanation of the fractal structure which arises when a large number of chips are placed at the origin and allowed to topple. The properties of the odometer play an important role in their analysis. In the literature of divisible sandpiles models, the scaling limit of the odometer was determined for an $\alpha$-stable divisible sandpile in \cite{frometa:jara}, who deal with a divisible sandpile for which mass is distributed not only to nearest-neighbor sites, but also to ``far away'' ones. Their limit is related to an obstacle problem for the truncated fractional Laplacian. In the subsequent work \cite{CHRheavy}, the authors of the present paper extend the result to the case in which the assumption on the finite variance of the $\sigma$'s is relaxed, and obtain an alpha-stable generalised field in the scaling limit.

The discrete bilaplacian (also called {\em membrane}) model was introduced in \cite{Sakagawa} and \cite{Kurt_d5,Kurt_d4} for the box of $\Z^d$ with zero boundary conditions. In $d\ge 4$ \cite{SunWu} and \cite{SunWu_d4} construct a discrete model for the bilaplacian field by assigning random signs to each component of the uniform spanning forest of a graph and study its scaling limit. As far as the authors know, \cite{LMPU} is the first paper in which the discrete bilaplacian model has been considered with periodic boundary conditions.

\subsection{Main results}\label{sec:main_res}
\paragraph{\it Notation.}We start with some preliminary notations which are needed throughout the paper. 
Let $\T^d$ be the $d$-dimensional torus, alternatively viewed as $\frac{\R^d}{\Z^d}$ or as $[-\frac12,\,\frac12)^d\subset\R^d$. $\Z_n^d:=[-\frac{n}{2},\,\frac{n}{2}]^d\cap \Z^d$ is the discrete torus of side-length $n\in \N$, and $\T_n^d:=[-\frac12,\,\frac12]^d\cap (n^{-1}\Z)^d$ is the discretization of $\T^d$.  Moreover let $B(z,\,\rho)$ a ball centered at $z$ of radius $\rho>0$ in the $\ell^\infty$-metric. We will use throughout the notation $z\cdot w$ for the Euclidean scalar product between $z,\,w\in \R^d$. With $\|\,\cdot\,\|_\infty$ we mean the $\ell^\infty$-norm, and with $\|\cdot\|$ the Euclidean norm. We will let $C,\,c$ be positive constants which may change from line to line within the same equation. We define the Fourier transform of a function $u\in L^1(\T^d)$ as $\widehat u(y):=\int_{\T^d}u(z)\exp\left(-2\pi\i y\cdot z\right)\De z$ for $y\in \Z^d$. We will use the symbol $\,\widehat \cdot\,$ to denote also Fourier transforms on $\Z_n^d$ and $\R^d$. We will say that a function $f(n)=\o{1}$ if $\lim_{n\to+\infty}f(n)=0$.

We can now state our main theorem: we consider the piecewise interpolation of the odometer on small boxes of radius $\frac{1}{2n}$ and show convergence to the continuum bilaplacian field.
\begin{theorem}[Scaling limit of the odometer for Gaussian weights]\label{thm:main}
Let $d\ge 1$ and let $(\sigma(x))_{x\in \Z_n^d}$ be a collection of i.i.d. standard Gaussians. Let $e_n(\cdot):=e_{\Z_n^d}(\cdot)$ be the odometer on $\Z_n^d$ associated to these weights. The formal field 
\begin{equation}\label{eq:deffield}
\Xi_n(x):=4\pi^2\sum_{z\in \T^d_n}n^{\frac{d-4}{2}}e_n({nz})\one_{B\left(z,\,\frac{1}{2n}\right)}(x),\quad x\in \T^d
\end{equation}
converges in law as $n\to+\infty$ to the bilaplacian field $\Xi$ on $\T^d$. The convergence holds in the Sobolev space $\mathcal H_{-\eps}(\T^d)$ with the topology induced by the norm $\|\cdot\|_{\mathcal H_{-\eps}(\T^d)}$ for any $\eps>\max\left\{1+\frac{d}{4},\,\frac{d}{2}\right\}$ (see Section~\ref{subsec:review} for the analytic specifications).
\end{theorem}
The reason to impose $\eps>\max\left\{1+\frac{d}{4},\,\frac{d}{2}\right\}$ is two-folded: on the one hand, it ensures the tightness of $\Xi_n$, on the other it allows us to define the law of $\Xi$ properly (see the construction of abstract Wiener space in Subsection~\ref{subsec:review}). Observe moreover that $\max\left\{1+\frac{d}{4},\,\frac{d}{2}\right\}$ has a transition at $d=4$, which is reminiscent of the phase transition of the bilaplacian model on $\Z^d$ (see for instance \cite{Kurt_d4}).

We can now show the next Theorem, which generalises the previous one to the case in which the weights have an arbitrary distribution with mean zero and finite variance. We keep the proof separate from the Gaussian one, as the latter will allow us to obtain precise results on the kernel of the bilaplacian, and has also a different flavor. Moreover, the more general proof relies on estimates we obtain in the Gaussian case. With a slight abuse of notation, we will define a field $\Xi_{n}$ as in Theorem~\ref{thm:main} also for weights which are not necessarily Gaussian (in the sequel, it will be clear from the context to which weights we are referring to).
\begin{theorem}[Scaling limit of the odometer for general weights]\label{thm:3}
Assume $(\sigma(x))_{x\in \Z_n^d}$ is a collection of i.i.d. variables with $\E\left[\sigma\right]=0$ and $\E\left[\sigma^2\right]=1$. Let $d\ge 1$ and $e_n(\cdot)$ be the corresponding odometer. If we define the formal field $\Xi_n$ as in \eqref{eq:deffield} for such weights, then it
converges in law as $n\to+\infty$ to the bilaplacian field $\Xi$ on $\T^d$. The convergence holds in the same fashion of Theorem~\ref{thm:main}.
\end{theorem}
We now give an explicit description of the covariance structure of $\Xi$. Our motivation is also a comparison with the whole-space bilaplacian field already treated in the literature. More precisely, for $d\ge 5$, \citet[Definition~3]{SunWu} define the bilaplacian field $\widetilde \Xi_d$ on $\R^d$ as the unique distribution on $\left(C^\infty_c(\R^d)\right)^\ast$ such that, for all $u\in C^\infty_c(\R^d)$, $\la\widetilde \Xi_d,\,u \ra$ is a centered Gaussian variable with variance
\[
\E\left[\la\widetilde \Xi_d,\,u \ra^2\right]=\iint_{\R^d\times\R^d}u(x)u(y)\|x-y\|^{4-d}\De x \De y.
\]
Since we obtain a limiting field on $\T^d$, we
think it is interesting to give a representation for the covariance kernel of the biharmonic operator in our setting. From now on, when we use the terminology ``zero average'' for a function $u$, we always mean $\int_{\T^d}u(x)\De x=0.$
\begin{theorem}[Kernel of the biharmonic operator in higher dimensions]\label{corol:kernel}
Let $d\ge 5$. Let furthermore $u \in C^\infty(\T^d)$ and with zero average. Then there exists $\mathcal G_d\in L^1(\R^d)$ such that
\begin{align}
\E\left[\la \Xi_,\,u \ra^2\right]&=\left(u,\,\Delta^{-2}u\right)_{L^2(\T^d)}\nonumber\\
&=\iint_{\T^d\times \T^d} u(z)u(z')\sum_{w\in \Z^d}\mathcal G_d(z-z'+w)\De z\De z'.\label{eq:kernel_exp}
\end{align}
$\mathcal G_d$ can be computed as follows: 
there exists $h_d\in C^\infty(\R^d)$ depending on $d$ such that
\eq{}\label{eq:gexp}
\mathcal G_d(\,\cdot\,)=
{\pi^{4-\frac{d}{2}}\Gamma\left(\frac{d-4}{2}\right)}\|\cdot\|^{4-d}+h_d(\,\cdot\,).
\eeq{}
\end{theorem}
\begin{remark}[Kernel of the biharmonic operator in lower dimensions]\label{rem:lower}
The convergence result of Theorem~\ref{thm:3} allows us to determine the kernel in $d\le 3$ too.  In fact, for such $d$ interchanging sum and integrals is possible, so that we can write
\begin{equation}\label{eq:kernel}
\left(u,\,\Delta^{-2}u\right)_{L^2(\T^d)}=\sum_{\nu\in \Z^d\setminus\{0\}}\frac{\left|\widehat u(\nu)\right|^2}{\|\nu\|^{4}}=\iint_{\T^d\times \T^d}u(z)u(z')\mathcal K(z-z')\De z\De z',
\end{equation}
where we can define the kernel of the bilaplacian to be
\[
\mathcal K(z-z'):=\sum_{\nu\in \Z^d\setminus\{0\}}\frac{\e^{2\pi\i (z-z')\cdot \nu}}{\|\nu\|^{4}},\quad z,\,z'\in \T^d.
\]
\end{remark}
\paragraph{\it Outline of the article.}The necessary theoretical background is given in Section~\ref{sec:prel}, together with an outline of the strategy of the proof of Theorem~\ref{thm:main}. Auxiliary results and estimates are provided in Section~\ref{sec:proof_prop}. The proof of Theorem~\ref{thm:main} lies  in Section~\ref{sec:proof_main}, and of Theorem~\ref{thm:3} in Section~\ref{sec:thm3}. Finally we conclude with the proof of Theorem~~\ref{corol:kernel} in Section~\ref{sec:corol}.
\paragraph{\it{Acknowledgments}}We would like to thank Xin Sun for pointing out to us the paper \cite{SunWu_d4}. We are grateful to Swagato K. Ray and Enrico Valdinoci for helpful discussions, and to an anonymous referee who helped in improving and clarifying the paper. The first author's research was partially supported by the Dutch stochastics cluster STAR (Stochastics -- Theoretical and Applied Research) and by the EPSRC grant EP/N004566/1. The second author's research was supported by Cumulative Professional Development Allowance from Ministry of Human Resource Development, Government of India and Department of Science and Technology, Inspire funds.
\section{Preliminaries}\label{sec:prel}
In this section we review the basics of the spectral theory of the Laplacian on the discrete torus from \cite{LMPU}. We also remind the fundamentals of abstract Wiener spaces which enable us to construct standard Gaussian random variables on a Sobolev space on $\T^d$. The presentation is inspired by \cite{Silvestri}. We also comment on the basic strategy of the proof of Theorem~\ref{thm:main} and make some important remarks on the test functions we use for our calculations. We refer for the Fourier analytic details used in this article to \cite{stein:weiss} and for a survey on random distributions to \cite{GVbook}.
\subsection{Fourier analysis on the torus}
We now recall a few facts about the eigenvalues of the Laplacian from \cite{LMPU} for completeness. Consider the Hilbert space $L^2(\Z_n^d)$ of complex valued functions on the discrete torus endowed with the inner product
$$\la f, g \ra = \frac1{n^d} \sum_{x\in \Z_n^d} f(x)\overline{g(x)}.$$
The Pontryagin dual group of $\Z_n^d$ is identified again with $\Z_n^d$. Let $\{ \psi_a: a\in \Z_n^d\}$ denote the characters of the group where $\psi_a(x)= \exp(2\pi \i x\cdot \frac{a}{n})$. 
The eigenvalues of the Laplacian $\Delta_g$ on discrete tori are given by $$\lambda_w= -4 \sum_{i=1}^d \sin^2\left(\frac{\pi w_i}{n}\right),\quad w\in\Z_n^d.$$
Recalling \eqref{eq:g}, we use the shortcut $g_x(y):= g(y,x)$.
Let $\widehat g_x$ denote the Fourier transform of $g_x$. It follows that
\begin{equation}\label{eq:defL}
\widehat g_x(0)= n^{-d}\sum_{y\in \Z_n^d} g_x(y) =:L 
\end{equation}
for all $x\in \Z_n^d$ (it can be seen in several ways, for example by translation invariance, that $L$ is independent of $x$). Finally, we recall \citet[Equation (20)]{LMPU}: for all $a\neq 0$,
\eq{}\label{eq:20}
\lambda_a \widehat{g_x}(a)=-2d n^{-d}\psi_{-a}(x).
\eeq{}
\subsection{Gaussian variables on homogeneous Sobolev spaces on the torus}\label{subsec:review}
Since our conjectured scaling limit is a random distribution, we think it is important to keep the article self-contained and give a brief overview of analytic definitions needed to construct the limit in an appropriate functional space. Our presentation is based on \citet[Section 2]{Sheff} and \citet[Sections~6.1, 6.2]{Silvestri}.

An {\em abstract Wiener space} (AWS) is a triple $(H, B, \mu)$, where:
\begin{enumerate}
\item $(H, (\cdot,\cdot)_H)$ is a Hilbert space,
\item $(B, \|\cdot\|_{B})$ is the Banach space completion of $H$ with respect to the measurable norm $\|\cdot\|_B$ on $H$, equipped with the Borel $\sigma$-algebra $\mathcal B$ induced by $\|\cdot\|_B$, and
\item $\mu$  is the unique Borel probability measure on $(B,\mathcal B)$ such that, if $B^*$ denotes the dual space of $B$, then $\mu\circ\phi^{-1}\sim \mathcal N(0, \|\widetilde \phi\|^2_{H}) $ for all $\phi\in B^*$, where $\widetilde \phi$ is the unique element of $H$ such that $\phi(h)=(\widetilde \phi, h)_H$ for all $h\in H$.
\end{enumerate}
We remark that, in order to construct a measurable norm $\|\cdot\|_B$ on $H$, it suffices to find a Hilbert- Schmidt operator $T$ on $H$, and set $\|\cdot\|_B :=\|T\cdot\|_H$.

Let us construct then an appropriate AWS. Choose $a\in \R$. Let us define the operator $(-\Delta)^a$ acting on $L^2(\T^d)$-functions $u$ with Fourier series $\sum_{\nu\in \Z^d}\widehat u(\nu)\mathbf e_{\nu}(\cdot)$ as follows ($(\mathbf e_\nu)_{\nu\in \Z^d}$ denotes the Fourier basis of $L^2(\T^d)$):
\[
(-\Delta)^a \left(\sum_{\nu\in \Z^d}\widehat u(\nu)\mathbf e_{\nu}\right)(\vartheta)=\sum_{\nu\in \Z^d\setminus\{0\}}\|\nu\|^{2a}\widehat u(\nu)\mathbf e_{\nu}(\vartheta).
\]
Let ``$\sim$'' be the equivalence relation on $C^\infty(\T^d)$ which identifies two functions differing by a constant and let $H^a(\T^d)$ be the Hilbert space completion of $C^\infty(\T^d)/{\sim}$ under the norm
\[
(f,\,g)_{a}:=\sum_{\nu\in \Z^d\setminus\{0\}}\|\nu\|^{4a}\widehat f(\nu)\widehat g(\nu).
\]
Define the Hilbert space 
\[
\mathcal H_{a}:=\left\{u\in L^2(\T^d):\,(-\Delta)^a u\in L^2(\T^d)\right\}/{\sim}.
\]
We equip $\mathcal H_{a}$ with the norm
\[
\|u\|_{\mathcal H_{a}(\T^d)}^2=\left((-\Delta)^a u,\,(-\Delta)^a u\right)_{L^2(\T^d)}.
\]
In fact, $(-\Delta)^{-a}$ provides a Hilbert space isomorphism between $\mathcal H_{a}$ and $H^a(\T^d)$, which when needed we identify. 
For 
\eq{}\label{eq:norm_meas}
b<a-\frac{d}{4}
\eeq{}
one shows that $(-\Delta)^{b-a}$ is a Hilbert-Schmidt operator on $H^a$ (cf.~also \citet[Proposition~5]{Silvestri}). In our case, we will be setting $a:=-1$. Therefore, by \eqref{eq:norm_meas},
for any $-\epsilon:=b<0$ which satisfies $\eps>1+\frac{d}{4}$, we have that $(H^{-1},\,\mathcal{H}_{-\epsilon},\,\mu_{-\epsilon})$ is an AWS. The measure $\mu_{-\epsilon}$ is the unique Gaussian law on $\mathcal H_{-\eps}$ whose characteristic functional is
\[
\Phi(u):=\exp\left(-\frac{\|u\|_{{-1}}^2}{2}\right).
\]
The field associated to $\Phi$ will be called $\Xi$
and is the limiting field claimed in Theorem~\ref{thm:main}.

There is a perhaps more explicit description of $\Xi$ which is based on {\em Gaussian Hilbert spaces} \cite[Chapter 1]{Jan97}. The construction is taken from \citet[Example~1.25]{Jan97}. Let $(\Omega,\,\mathcal A,\,P)$ be a probability space with $\mathcal A$ its Borel $\sigma$-algebra. Assume that on $\Omega$ one can define a sequence of i.i.d.\ standard Gaussians $({X}_m)_{m\in \N}$. Let further $(\mathbf{X}_m)_{m\in \N}$ be an orthonormal basis of $H^{-1}(\T^d)$. Then there is an isometric embedding $ \Xi:\,H^{-1}(\T^d)\hookrightarrow L^2(\Omega,\,P)$ such that $\la\Xi,\,\mathbf X_m\ra\stackrel{d}{=}X_m$ for all $m$. Indeed, by the properties of AWS, the mapping $(H^{-\eps})^\ast\ni \phi \mapsto \la \Xi,\,\phi\ra$ is an isometry of the dense subspace $(H^{-\eps})^\ast$ onto $S:=\left\{\la\Xi,\,u\ra:\,u\in   (H^{-\eps})^\ast\right\}$. The mapping can be extended by continuity to an isometry between $H^{-1}$ and the corresponding closure of $S$. Taking $\Omega:=\mathcal H_{-\eps}$ and $P:=\mu_{-\eps}$, this entails an alternative construction of $\Xi$: it is the unique Gaussian process indexed by $H^{-1}$ such that $\Xi\stackrel{d}{=}\left\{\la\Xi,\,u\ra:\,u\in   H^{-1}(\T^d)\right\}$ with $\la\Xi,\,u\ra\sim \mathcal N\left(0,\,\|u\|_{-1}^2\right)$ for any $u\in H^{-1}(\T^d).$
\subsection{Strategy of the proof of Theorem~\ref{thm:main}.} Firstly, we show that $\eta$ can be decomposed into the sum of two independent fields, namely 
\begin{proposition}\label{prop:decomposition}
There exist a centered Gaussian field $(\chi_x)_{x\in \Z_n^d}$ with covariance $\E[\chi_x\chi_y]=H(x,y)$ as in \eqref{eq:defH} and a centered normal random variable $Y$ with variance $(2d)^{-2} n^d L^2$ (where $L$ is as in~\eqref{eq:defL}), such that $Y$ is independent from $(\chi_x)_{x\in \Z_n^d}$ and 
$$(\eta(x))_{x\in \Z_n^d}\overset{d}=(Y+\chi_x)_{x\in\Z_n^d}.$$
In particular, $e_n(\cdot)$ admits the representation
 $$(e_n(x))_{x\in \Z_n^d} \overset{d}=\left(\chi_x- \min_{z\in \Z_n^d} \chi_z\right)_{x\in\Z_n^d}.$$
\end{proposition}
This decomposition is similar in spirit to the one in the proof of \citet[Proposition 1.3]{LMPU}, but we stress that the random fields we find are different. The proof of the above Proposition can be found in Subsection~\ref{subsec:show_prop}. As a consequence, to achieve Theorem~\ref{thm:main} it will suffice to determine the scaling limit of the $\chi$ field, because test functions have zero average, and hence we can get rid of the minimum appearing in the odometer representation. We will therefore show
\begin{itemize}
\item[(P1)] $\left(\mathcal L(\Xi_n)\right)_{n\in \N}$ is tight in the space $\mathcal H_{-\epsilon}(\T^d)$ where $-\epsilon<-\frac{d}{2}$.
\item[(P2)]\label{one_spade} From the above tightness result, there exists a subsequential scaling limit $\Xi=\lim_{k\to+\infty}\Xi_{n_k}$ for the convergence in law in the space $\mathcal H_{-\epsilon}$. The proof is complete once we show this limit is unique: by \citet[Section~2.1]{ledoux:talagrand}, it suffices to prove that, for all mean-zero test functions $u\in C^\infty(\T^d)$,
\[
\lim_{n\to+\infty}\E\left[\exp\left(\i\la \Xi_n,\,u\ra\right)\right]=\Phi(u),
\]
where the RHS is the characteristic function of $\Xi$. We will calculate the limit of the second moment of $\la \Xi_n,\,u\ra$ directly in $d\le 3$ and through a mollifying procedure in $d\ge 4$.
\end{itemize}
This will conclude the proof. Since the ``finite dimensional'' convergence is somewhat more interesting, we will defer the tightness proof to Subsection~\ref{subsec:tightness} and show (P2) in Subsection~\ref{subsec:marginal}.
\paragraph{\it A note on test functions.}
By the above construction, the set of test functions we will consider is the set of smooth functions $C^\infty(\mathbb T^d)$ with zero mean. We need to stress at this juncture an important remark: $C(\T^d)$ does {\em not} correspond to the class of continuous functions on $[-\frac12, \,\frac12)^d$, but only to functions which remain continuous on $\R^d$ when extended by periodicity. Similar comments apply to $C^{\infty}(\T^d)$ functions. 
See also \citet[Section 1, Chapter VII]{stein:weiss} for further discussions. {Therefore, when we consider $u: \R^d \to \R$ which is periodic and belongs to $C^\infty$, we consider its restriction to $[-\frac12, \,\frac12)^d$} while computing its integral on $\T^d$.
\section{Auxiliary results}\label{sec:proof_prop}
In this section we provide a proof of Proposition~\ref{prop:decomposition}. The result helps us tackle the singularity arising from the zero eigenvalue of $\Delta_g$ and will also reduce the determination of the scaling limit to finding the scaling limit of $(\chi_x)_{x\in \Z_n^d}$.

\subsection{Proof of Proposition~\ref{prop:decomposition}}\label{subsec:show_prop}
\begin{proof}
First, observe that, by Parseval's identity on the discrete torus, we can write the covariance of the Gaussian field $(\eta(x))_{x\in \Z_n^d}$ as
\begin{align}
\E&\left[ \eta(x) \eta(y)\right]=(2d)^{-2}\sum_{z\in \Z_n^d}g(z,\,x)g(z,\,y)\nonumber
\\&=(2d)^{-2}n^{d}\widehat{g_x}(0)\widehat{g_y}(0)+(2d)^{-2}n^{d}\sum_{z\in \Z_n^d\setminus\{0\}}\widehat{g_x}(z)\overline{\widehat{g_y}(z)}.\label{eq:splitting_sum}
\end{align}
First observe that using the description of $g(x,y)$ in terms of the simple random walk $(S_m)_{m\ge 0}$ on $\Z_n^d$ we derive
\begin{align}
\widehat{g_x}(0)&= n^{-d}\sum_{y\in \Z_n^d} g_x(y)=n^{-2d}\sum_{y\in \Z_n^d} \sum_{z\in \Z_n^d}\sum_{m\ge0} \mathsf P_x( S_m=y, m< \tau_z) \nonumber\\
&= n^{-2d}\sum_{z\in \Z_n^d} \sum_{y\in \Z_n^d\setminus\{z\}} \sum_{m\ge 0}\mathsf P_x(S_m=y, m< \tau_z)\nonumber\\
&= n^{-2d} \sum_{z\in \Z_n^d}\sum_{m\ge 0} \mathsf P_x(\tau_z>m)= n^{-2d} \sum_{z\in \Z_n^d}\E_x[ \tau_z].\label{eq:RTL}
\end{align}
One can notice that $\widehat{g_x}(0)$ is independent of $x$ by translation invariance. Hence we get that the first term in the left-hand side of \eqref{eq:splitting_sum} is a constant equal to
 $(2d)^{-2}n^d L^2$ having set $L:=n^{-2d}\sum_{q\in \Z_n^d}\mathsf E_x[\tau_q]$. As for the contribution from other sites,
\begin{align*}
(2d)^{-2}n^{d}\sum_{z\in \Z_n^d\setminus\{0\}}\widehat{g_x}(z)\overline{\widehat{g_y}(z)}&\stackrel{\eqref{eq:20}}{=}n^{-d}\sum_{z\in \Z_n^d\setminus\{0\}}\frac{\exp\left(-2\pi \i x\cdot \frac{z}{n}\right)\exp\left(2\pi \i y\cdot \frac{z}{n}\right)}{|\lambda_z|^2}.
\end{align*}
Define a centered Gaussian field $(\chi_x)_{x\in \Z_n^d}$ with covariance given by
\begin{equation}\label{eq:defH}
H(x,y)=\frac{n^{-d}}{16} \sum_{z\in \Z_n^d\setminus\{0\}} \frac{\exp(2\pi \i(y-x)\cdot \frac{z}{n})}{\left(\sum_{i=1}^d \sin^2\left(\pi \frac{z_i}{n}\right)\right)^2}.
\end{equation}
The field associated to $H$ is well-defined and in fact $H$ is positive definite. To see this, given a function $c:\Z_n^d\to \mathbb C$ one has that $\sum_{x,y\in \Z_n^d} H(x,y) c(x)\overline{c(y)} \ge 0$. Indeed,
\begin{align*}
&\sum_{x,y\in \Z_n^d} H(x,y) c(x)\overline{c(y)} =\frac{n^{-d}}{16}\sum_{x,y\in \Z_n^d}\sum_{z\in \Z_n^d\setminus\{0\}} \frac{\exp(2\pi (y-x)\cdot \frac{z}{n})}{\left(\sum_{i=1}^d \sin^2\left(\pi \frac{z_i}{n}\right)\right)^2}c(x)\overline{c(y)}\\
&=\frac{n^{-d}}{16}\sum_{z\in \Z_n^d\setminus\{0\}} d(z)\overline{d(z)}\ge 0,
\end{align*}
where $d(z):=\sum_{x\in \Z_n^d} {\exp(-2\pi \i x\cdot \frac{z}{n})}{{\left(\sum_{i=1}^d \sin^2(\pi \frac{z_i}{n})\right)^{-1}}}c(x).$ 
Hence it turns out that $(\eta(x))_{x\in \Z_n^d}$ has the same distribution as $(Y+\chi_x)_{x\in \Z_n^d}$ where $Y$ is a Gaussian random variable with mean zero and variance $(2d)^{-2} n^dL^2$ independent of the field $\chi$. To conclude, note that the odometer function satisfies $e_n(x)\stackrel{d}{=} \eta(x)-\min_{z\in \Z_n^d} \eta(z)\overset{d}= \chi_x-\min_{z\in\Z_n^d} \chi_z$. 
\end{proof}
\section{Proof of Theorem~\ref{thm:main}}\label{sec:proof_main}
We recall that it will suffice to prove the two properties (P1) and (P2) to achieve the Theorem. We first use to our advantage the fact that the test functions we consider have zero average, hence we can get rid of the minimum term which appears in the definition of the odometer.  Let us recall the field in \eqref{eq:deffield} $$\Xi_n(\cdot)=4\pi^2\sum_{z\in \T^d_n}n^{\frac{d-4}{2}}e_n({nz})\one_{B\left(z,\,\frac{1}{2n}\right)}(\cdot).$$
We define a linear functional on $C^\infty(\T^d)$ by setting
\[
\la\Xi_n,\,u \ra:=\int_{\T^d}\left(4\pi^2 n^{\frac{d-4}{2}}\sum_{z\in \T^d_n}\one_{B\left(z,\,\frac{1}{2n}\right)}(x)e_n({nz})\right)u(x)\De x.
\]
However using Proposition~\ref{prop:decomposition}, and the fact that $u$ has zero mean, one sees that
\begin{align*}
&\la\Xi_n,\,u \ra=4\pi^2\sum_{z\in \T_n^d} n^{\frac{d-4}{2}}\chi_{nz}\int_{B(z,\frac{1}{2n})}  u(x)\De x-4\pi^2\sum_{z\in \T_n^d} n^{\frac{d-4}{2}}\left(\min_{w\in\Z_n^d}\chi_{w}\right)\int_{B(z,\frac{1}{2n})}  u(x)\De x\\
&=4\pi^2\sum_{z\in \T_n^d} n^{\frac{d-4}{2}}\chi_{nz}\int_{B(z,\frac{1}{2n})}  u(x)\De x=\la \Xi'_n,\,u\ra
\end{align*}
letting 
$$
\Xi_n'(\cdot):=4\pi^2\sum_{z\in \T^d_n}n^{\frac{d-4}{2}}\chi_{nz}\one_{B\left(z,\,\frac{1}{2n}\right)}(\cdot)
$$
By the theory of Gaussian Hilbert spaces of Subsection~\ref{subsec:review}, $\Xi_n=\Xi_n'$ in distribution. Hence in the sequel we will, with a slight abuse of notation, consider $\Xi_n'$ but denote it simply as $\Xi_n$, since the law of the two fields is the same.
We are now ready to begin with (P2).  
\subsection{Proof of {(P2)}}\label{subsec:marginal}
\paragraph{\it Overview of the proof.}We have just seen that
\[
\la\Xi_n,\,u \ra=4\pi^2\sum_{z\in \T_n^d} n^{\frac{d-4}{2}}\chi_{nz}\int_{B(z,\frac{1}{2n})}  u(x)\De x.
\]
We now replace the integral over the ball above by the value at its center and gather the remaining error term. More precisely we get
\begin{align*}
&4\pi^2\sum_{z\in \T_n^d} n^{\frac{d-4}{2}}\chi_{nz}\int_{B(z,\frac{1}{2n})}  u(x)\De x=4\pi^2\sum_{z\in \T_n^d} n^{\frac{d-4}{2}}\chi_{nz}n^{-d}\int_{B(z,\frac{1}{2n})}  n^d u(x)\De x\\
&=4\pi^2\sum_{z\in \T_n^d} n^{\frac{d-4}{2}}\chi_{nz}n^{-d}u(z) + 4\pi^2\sum_{z\in \T_n^d} n^{\frac{d-4}{2}}\chi_{nz}n^{-d}\left(\int_{B(z,\frac{1}{2n})}n^d u(x)\De x-u(z)\right)\\
&=4\pi^2 n^{-\frac{d+4}{2}}\sum_{z\in \T^d_n}\chi_{nz}u(z)+R_n(u).
\end{align*}
Here the remainder $R_n(u)$ is defined by
\eq{}\label{eq:rem_R_n}R_n(u):=4\pi^2\sum_{z\in \T_n^d} n^{\frac{d-4}{2}}\chi_{nz}n^{-d}\left(\int_{B(z,\,\frac{1}{2n})}n^d u(x)\De x-u(z)\right)=4\pi^2 n^{-\frac{d+4}{2}}\sum_{z\in \T_n^d}\chi_{nz} K_n(z)\eeq{}
where using that the volume of $B(z, \frac1{2n})$ is $n^{-d}$ we have
\eq{}\label{eq:K_n}K_n(z):=\int_{B(z,\,\frac{1}{2n})}n^d u(x)\De x-u(z)=n^d\left[\int_{B(z,\,\frac{1}{2n})}\left(u(x)-u(z)\right)\De x\right].\eeq{}
We observe that using the above decomposition one can split the variance of $\la\Xi_n,\,u\ra$ as 
\begin{align*}
\E\left[\la\Xi_n,\,u\ra^2 \right]&=16\pi^4 n^{-(d+4)}\sum_{z,\,z'\in \T_n^d} u(z)u(z') \E[ \chi_{nz}\chi_{nz'}]+\E\left[R_n(u)^2\right]\\
&+4\pi^2\E\left[n^{-\frac{d+4}{2}}\sum_{z\in \T_n^d} u(z)\chi_{nz}R_n(u)\right].
\end{align*}
To deal with the convergence of the above terms we need two  propositions. The first one shows that the first term yields the required limiting variance.
\begin{proposition}\label{lem:bhindi} In the notation of this Section,
\begin{align*}
16\pi^4\lim_{n\to+\infty}n^{-(d+4)}\sum_{z,\,z'\in \T_n^d} u(z)u(z') \E[ \chi_{nz}\chi_{nz'}]&=16\pi^4\lim_{n\to+\infty}n^{-(d+4)}\sum_{z,\,z'\in \T_n^d} u(z)u(z') H\left( {nz},\,{nz'}\right)\\
&=\|u\|^2_{{-1}}.
\end{align*}
\end{proposition}
The second Proposition says the remainder term is small. 
\begin{proposition}\label{lem:small}
In the notations of this Section, $\lim_{n\to+\infty}R_n(u)=0$ in $L^2$.
\end{proposition}
Then an application of the Cauchy-Schwarz inequality will allow us to deduce that
\[
\lim_{n\to+\infty}\E\left[\la \Xi_n,\,u\ra^2 \right]=\|u\|^2_{{-1}}
\]
and the condition (P2) will be ensured.
We give the proof of Proposition~\ref{lem:bhindi}, which is the core of our argument, in Subsection~\ref{subsubsec:marginal} and of Proposition~\ref{lem:small} in Subsection~\ref{subsubsec:rem}.
\subsubsection{{Proof of Proposition~\ref{lem:bhindi}}}\label{subsubsec:marginal}
Before we begin our proof we would like to prove a bound which would be crucial in estimating  the eigenvalues of the Laplacian on the discrete torus. This lemma will be used later for other parts of the proof too.
\begin{lemma}
There exists $c>0$ such that for all $n\in \N$ and $w\in \Z_n^d\setminus\{0\}$  we have
\begin{align}
\frac{1}{\|\pi w\|^4}\le n^{-4}\left(\sum_{i=1}^d \sin^2\left(\frac{\pi w_i}{n}\right)\right)^{-2}\le \left(\frac{1}{\|\pi w\|^2}+\frac{c}{n^{2}}\right)^2\label{eq:flash_gordon_bleu}
\end{align}
\end{lemma} 
\begin{proof}
We consider
\begin{align*}
\sum_{i=1}^d&n^2\sin^2\left(\frac{\pi w_i}{n}\right)=\sum_{i=1}^d w_i^2\pi^2\left(\frac{\sin\left(\theta_i^n\right)}{\theta_i^n}\right)^2
\end{align*}
with $\theta_i^n:=\pi{ w_i}{n^{-1}}\in \left[-{\pi}/{2},\,{\pi}/{2}\right]\setminus\{0\}.$
This gives the left-hand side of \eqref{eq:flash_gordon_bleu}. Moreover
\begin{align*}
\left\|\pi w\right\|^2-\sum_{i=1}^d n^2\sin^2\left(\frac{\pi w_i}{n}\right)=\sum_{i=1}^d w_i^2\pi^2\left(1-\left(\frac{\sin\left(\theta_i^n\right)}{\theta_i^n}\right)^2\right)\le C\|w\|^4 n^{-2}
\end{align*}
because $0\le 1-{\sin^2(x)}{x^{-2}}\le C\,x^{2}$ for some $C>0$. In this way
\begin{align}
\frac{1}{\sum_{i=1}^d n^2\sin^2\left(\frac{\pi w_i}{n}\right)}&-\frac{1}{\|\pi w\|^2}=\frac{\|\pi w\|^2-\sum_{i=1}^d n^2\sin^2\left(\frac{\pi w_i}{n}\right)}{\sum_{i=1}^d n^2\sin^2\left(\frac{\pi w_i}{n}\right)\|\pi w\|^2}\nonumber\\
&\le \frac{C\|w\|^4n^{-2}}{\sum_{i=1}^d n^2\sin^2\left(\frac{\pi w_i}{n}\right)\|\pi w\|^2}.\label{eq:refer}
\end{align}
Considering that, for $x\in \left[-{\pi}/{2},\,{\pi}/{2}\right]$, ${\sin^2(x)}{x^{-2}}\in\left[{4}/{\pi^2},\,1\right]$, one gets that
\begin{equation}\label{eq:mongoose}
\sum_{i=1}^d n^2\sin^2\left(\frac{\pi w_i}{n}\right)\ge {4}\|w\|^2
\end{equation}
which plugged into \eqref{eq:refer} gives that 
$$
\frac{1}{\sum_{i=1}^d n^2\sin^2\left(\frac{\pi w_i}{n}\right)}-\frac{1}{\|\pi w\|^2}\le C{n^{-2}}
$$
for $C>0$, thus \eqref{eq:flash_gordon_bleu} is proven.
\end{proof}
\begin{remark} The equation \eqref{eq:mongoose} is not enough to obtain sharp asymptotics for $\sum_{i=1}^d n^2 \sin^2\left(\pi{w_i}/{n}\right)$ as $n\to\infty$. On the other hand, we will use it in the sequel while looking for a uniform lower bound for the same quantity for all $w\neq 0$.
\end{remark}
We begin with the proof of Proposition~\ref{lem:bhindi}.
Let $u : \T^d\to \R$ be a smooth function with zero mean. Define $u_n:\Z_n^d\to \R$ as $u_n(z):= u(\frac{z}{n})$. Note that
\begin{align}
&16\pi^4 n^{-2d} n^{d-4}\sum_{z,\,z'\in \T_n^d} u(z)u(z') \E[ \chi_{nz}\chi_{nz'}]= 16\pi^4 n^{-2d}n^{d-4}  \sum_{z,\,z'\in\Z_n^d} u(z) u(z') H(nz, nz')\nonumber\\
&=\pi^4 n^{-2d}n^{-4} \sum_{z,\,z'\in\T_n^d} u(z) u(z') \sum_{w\in\Z_n^d\setminus\{0\}}\frac{\exp(2\pi \i(z-z')\cdot w)}{\left(\sum_{i=1}^d\sin^2\left(\frac{\pi w_i}{n}\right)\right)^2}.\label{eq:mmt}
\end{align}
To show the above expression converges it is enough to consider the convergence of 
\begin{equation}\label{eq:change}
n^{-2d}\sum_{z,\,z'\in\T_n^d} u(z) u(z') \sum_{w\in\Z_n^d\setminus\{0\}}\frac{\exp(2\pi \i(z-z')\cdot w)}{\| w\|^4}.
\end{equation}
This can be justified by showing that~\eqref{eq:mmt} can be bounded above and below appropriately by 
\eqref{eq:change}. Now observing that 
\begin{equation}\label{eq:chuchutv}
n^{-2d}\sum_{z,\,z'\in\T_n^d} u(z) u(z') \exp(2\pi \i(z-z')\cdot w)= \left|\widehat{ u_n}(w)\right|^2\ge 0
\end{equation}
the lower bound of \eqref{eq:flash_gordon_bleu} immediately gives
\begin{align*}
\pi^4 n^{-2d-4} &\sum_{z,\,z'\in\T_n^d} u(z) u(z') \sum_{w\in\Z_n^d\setminus\{0\}}\frac{\exp(2\pi \i(z-z')\cdot w)}{\left(\sum_{i=1}^d\sin^2\left(\frac{\pi w_i}{n}\right)\right)^2}\\
 &\ge n^{-2d} \sum_{z,\,z'\in\T_n^d} u(z) u(z') \sum_{w\in\Z_n^d\setminus\{0\}}\frac{\exp(2\pi \i(z-z')\cdot w)}{\| w\|^4}.
\end{align*}
For the upper bound, using the bound in~\eqref{eq:flash_gordon_bleu}  we get
\begin{align*}
&\pi^4 n^{-2d-4} \sum_{z,\,z'\in\T_n^d} u(z) u(z') \sum_{w\in\Z_n^d\setminus\{0\}}\frac{\exp(2\pi \i(z-z')\cdot w)}{\left(\sum_{i=1}^d\sin^2\left(\frac{\pi w_i}{n}\right)\right)^2}\\
&\le \pi^4 n^{-2d} \sum_{z,\,z'\in\T_n^d} u(z) u(z') \sum_{w\in\Z_n^d\setminus\{0\}}{\exp(2\pi \i(z-z')\cdot w)}\left(\frac1{\| \pi w\|^{2}}+\frac{c}{n^{2}}\right)^2.
\end{align*}
Now we expand the square: the first term gives the correct upper bound as in \eqref{eq:change} and the other two terms are negligible. In fact we show firstly that
$$\lim_{n\to+\infty}  c n^{-2d}n^{-2} \sum_{z,\,z'\in\T_n^d} u(z) u(z') \sum_{w\in\Z_n^d\setminus\{0\}}\frac{\exp(2\pi \i(z-z')\cdot w)}{\| w\|^2}=0.$$
Using \eqref{eq:chuchutv} and Parseval's identity we get
\begin{align*}
&c n^{-2d}n^{-2} \sum_{z,\,z'\in\T_n^d} u(z) u(z') \sum_{w\in\Z_n^d\setminus\{0\}}\frac{\exp(2\pi \i(z-z')\cdot w)}{\| w\|^2}= c n^{-2}\sum_{w\in \Z_n^d\setminus\{0\}} \frac1{\| w\|^2} \left|\widehat{u_n}(w)\right|^2\\
&\overset{\|w\|\ge1}\le c n^{-2}\sum_{w\in \Z_n^d\setminus\{0\}} |\widehat{u_n}(w)|^2\le c n^{-2}\sum_{w\in \Z_n^d} \left|\widehat{u_n}(w)\right|^2\\
&= c n^{-2}n^{-d} \sum_{w\in\Z_n^d} \left|u\left(\frac{w}{n}\right)\right|^2=c n^{-2}\left(n^{-d}\sum_{w\in \T_n^d} |u(w)|^2\right).
\end{align*}
Since $n^{-d}\sum_{w\in \T_n^d} |u(w)|^2\to \int_{\T^d} |u(w)|^2\De w<+\infty$ we get that the second term converges to zero. Note that the same computation shows
$$n^{-2d} n^{-4}\sum_{z,\,z'\in\T_n^d} u(z) u(z') \sum_{w\in\Z_n^d\setminus\{0\}}\exp(2\pi \i(z-z')\cdot w)\le n^{-4}\left(n^{-d}\sum_{w\in \T_n^d} |u(w)|^2\right),$$
which again goes to zero as $n\to+ \infty$. So this shows that we can from now on concentrate on showing the convergence of \eqref{eq:change}.
We split now our proof, according to whether $d\le 3$ or $d\ge 4$. 
\paragraph{\it The case $d\le 3$.}In the first case, the argument is more straightforward: we rewrite  
\[
\eqref{eq:change}= \sum_{w\in\Z^d\setminus\{0\}}\|w\|^{-4}\one_{w\in \Z_n^d}\sum_{z\in\T_n^d} n^{-d}u(z)\exp(2\pi \i z\cdot w) \sum_{z'\in\T_n^d}n^{-d}u(z')\exp(-2\pi \i z'\cdot w).
\]
Since $\sum_{z\in\T_n^d} n^{-d}u(z)\exp(2\pi \i z\cdot w)$ is bounded above uniformly in $n$, and $\sum_{w\in \Z^d\setminus\{0\}}\|w\|^{-4}<+\infty$ in $d<4$, we can apply the dominated converge theorem and obtain
\[
\lim_{n\to+\infty}\eqref{eq:change}=\sum_{w\in\Z^d\setminus\{0\}}\|w\|^{-4}\left|\widehat u(w)\right|^2=\|u\|^2_{-1}
\]
which concludes the proof of (P2) for $d\le 3$.
\paragraph{\it The case $d\ge 4$.}\label{par:d>3} Here it is necessary to think of another strategy since $\sum_{w\in \Z^d}\|w\|^{-4}$ is not finite. 
Let $\phi\in \mathcal S(\R^d)$, the Schwartz space, be a mollifier supported on $[-\frac12, \,\frac12)^d$ with
$\int_{\R^d}\phi(x)\De x=1$ and let $\phi_\kappa(x):= \kappa^{-d}\phi(\frac{x}{\kappa})$ for $\kappa>0$. It is a classical result \cite[Theorem~7.22]{Rudin} that for $\delta=0,\,1,\,2\,\ldots$ there exists $A>0$ (depending on $\kappa$ and $\delta$) such that 
\eq{}\label{eq:decay_phi}
\left|\widehat{\phi_\kappa}(w)\right|\le A\left(1+\|w\|\right)^{-\delta} .
\eeq{}
Now to show the convergence of \eqref{eq:change} is equivalent to considering 
$$\lim_{\kappa\to 0}\lim_{n\to +\infty}n^{-2d}\sum_{z,\,z'\in\T_n^d} u(z) u(z') \sum_{w\in\Z_n^d\setminus\{0\}}\widehat{\phi_\kappa}(w)\frac{\exp(2\pi \i(z-z')\cdot w)}{\|w\|^4}$$
since we claim that
\begin{align}
&\lim_{\kappa\to 0}\limsup_{n\to +\infty}n^{-2d}\sum_{z,\,z'\in\T_n^d} u(z) u(z') \sum_{w\in\Z_n^d\setminus\{0\}}\left(\widehat{\phi_\kappa}(w)-1\right)\frac{\exp(2\pi \i(z-z')\cdot w)}{\|w\|^4}=0.\label{eq:claim0}
\end{align}
Indeed, using the fact that $\int_{\R^d} \phi_\kappa(x)\De x=1$ we have
$$\left|\widehat{\phi_\kappa}(w)-1\right|\le  \int_{\R^d} \phi_\kappa(y)\left|\e^{2\pi \i y\cdot w}-1\right|\De y.$$
Exploiting the fact that $|\exp(2\pi \i x)-1|^2= 4\sin^2(\pi x)$ and $|\sin(x)|\le |x|$ we obtain
\begin{align}
&\left|\widehat{\phi_\kappa}(w)-1\right|
\le C\kappa\|w\|\int_{\R^d} \|y\|\phi(y) \De y\le C \kappa\|w\|\label{eq:banjaara_two}
\end{align}
due to the fact that $\phi$ is supported on $[-\frac{1}{2},\,\frac{1}{2})^d$. Recalling $u_n(z)=u(\frac{z}{n})$ and plugging the estimate \eqref{eq:banjaara_two} in \eqref{eq:claim0} we get that
\begin{align}
&\left|n^{-2d}\sum_{w\in\Z_n^d\setminus\{0\}}\frac{\widehat{\phi_\kappa}(w)-1}{\|w\|^4}\sum_{z,\,z'\in\T_n^d} u(z) u(z') \exp(2\pi \i(z-z')\cdot w)\right|\nonumber\\
&\le C\kappa\sum_{w\in \Z_n^d\setminus \{0\}} \|w\|^{-3} \left|\widehat{u_n} (w)\right|^2.\label{bahbah}
\end{align}
Using $\|w\|\ge 1$ we have
\begin{align*}
& \sum_{w\in \Z_n^d\setminus \{0\}} \|w\|^{-3} \left|\widehat{u_{n}} (w)\right|^2\le \sum_{w\in \Z_n^d\setminus \{0\}} \left|\widehat{u_{n}} (w)\right|^2\le \sum_{w\in \Z_n^d} \left|\widehat{u_{n}} (w)\right|^2\\
&=n^{-d} \sum_{w\in \Z_n^d} \left|u\left(\frac{w}{n}\right)\right|^2 = n^{-d}\sum_{w\in \T_n^d} |u(w)|^2
\end{align*}
where we have used Parseval's identity. We observe then that 
\begin{align*}
&\limsup_{n\to+\infty}\left|n^{-2d}\sum_{w\in\Z_n^d\setminus\{0\}}\frac{\widehat{\phi_\kappa}(w)-1}{\|w\|^4}\sum_{z,\,z'\in\T_n^d} u(z) u(z') \exp(2\pi \i(z-z')\cdot w)\right|\\
&\le C\kappa\|u\|_{L^2(\T^d)}^2<+\infty.
\end{align*}
Taking the limit $\kappa\to 0$ in the previous expression we deduce the claim~\eqref{eq:claim0}.
Now we have to derive the limit of the following expression:
\begin{align}
&n^{-2d}\sum_{z,\,z'\in\T_n^d}  u(z) u(z') \sum_{w\in\Z_n^d\setminus\{0\}}\widehat{\phi_\kappa}(w)\frac{\exp(2\pi \i(z-z')\cdot w)}{\|w\|^4}.\label{eq:fish}
\end{align}
Since $\widehat{\phi_\kappa}$ has a fast decay at infinity, and 
$$\lim_{n\to+\infty}n^{-d}\sum_{z\in\T_n^d} u(z) \exp(2\pi \i z\cdot w)=\widehat u(w)$$
we can apply the dominated convergence theorem to obtain
\[
\lim_{n\to+\infty}n^{-2d}\sum_{z,\,z'\in\T_n^d} u(z) u(z') \sum_{w\in\Z_n^d\setminus\{0\}}\widehat{\phi_\kappa}(w)\frac{\exp(2\pi \i(z-z')\cdot w)}{\|w\|^4}=\sum_{w\in\Z^d\setminus\{0\}}\widehat{\phi_\kappa}(w)\frac{\left|\widehat{u}(w)\right|^2}{\|w\|^4}.
\]
The bound $|\widehat{\phi_\kappa}(\cdot)|\le 1$ can be used to obtain a bound uniform in $\kappa$ on the right-hand side of the above expression: consequently we apply the dominated convergence letting $\kappa\to 0$ to achieve
\[
\lim_{\kappa\to 0}\lim_{n\to+\infty}n^{-2d}\sum_{z,\,z'\in\T_n^d} u(z) u(z') \sum_{w\in\Z_n^d\setminus\{0\}}\widehat{\phi_\kappa}(w)\frac{\exp(2\pi \i(z-z')\cdot w)}{\|w\|^4}=\sum_{w\in\Z^d\setminus\{0\}}\frac{\left|\widehat{u}(w)\right|^2}{\|w\|^4}=\|u\|^2_{-1}.
\]
This concludes the proof of Proposition~\ref{lem:bhindi}.
\subsubsection{{Proof on the remainder: Proposition~\ref{lem:small}}}\label{subsubsec:rem}
We owe the reader now the last proofs on $R_n$ (see \eqref{eq:rem_R_n}). First we state the following
\begin{lemma}\label{lem:shave}
There exists a constant $C>0$ such that $\sup_{z\in \T^d}|K_n(z)|\le Cn^{-1}$.
\end{lemma}
\begin{proof}
Using the mean value theorem as $u\in C^\infty(\T^d)$ we get that, for some $c\in (0,1)$,
\begin{align*}
|K_n(z)| &\le n^d\int_{B\left(z,\frac1{2n}\right)} |u(x)-u(z)|\De x\le n^d\int_{B\left(z,\frac{1}{2n}\right)} \|\nabla u(cx+(1-c)z) \| \,\|z-x\|\De x\\
&\le C\frac{ n^d}{2n} \int_{B\left(z, \frac1{2n}\right)} \|\nabla u(cx+(1-c)z) \| \De x\le C \frac{\|\nabla u\|_{L^\infty(\T^d)}}{n}.
\end{align*}
Since $\|\nabla u\|_{L^\infty(\T^d)}<+\infty$ the claim follows.
\end{proof}
We reprise now the proof on the limit of $R_n(u).$
\begin{proof}[Proof of Proposition~\ref{lem:small}]
We first compute $\E\left[R_n(u)^2\right]$ obtaining
\begin{align*}
\E\left[R_n(u)^2\right]&=16\pi^4 n^{-2d}\sum_{z,\,z'\in \T^d_n}n^{d-4}H\left(nz,\,nz'\right)K_n(z)K_n\left(z'\right)\\
&\stackrel{\eqref{eq:mongoose}}{\le}{n^{-2d}} \sum_{z,\,z'\in \T^d_n}\sum_{w\in \Z_n^d\setminus\{0\}}\frac{\exp(2\pi \i (z-z')\cdot w)}{\|w\|^4}K_n(z)K_n\left(z'\right)\\
&\le n^{-2d} \sum_{z,\,z'\in \T^d_n}\sum_{w\in \Z_n^d\setminus\{0\}}{\exp(2\pi \i (z-z')\cdot w)}K_n(z)K_n\left(z'\right)
\end{align*}
since $\|w\|\ge 1$. Letting $K'_n(x): =K(\frac{x}{n})$, thanks to Lemma~\ref{lem:shave} we have that the previous expression is equal to
\begin{align*}
&\sum_{w\in \Z_n^d\setminus\{0\}}  \widehat{K'_n}(w)\overline{\widehat{K'_n}(w)} \leq \sum_{w\in \Z_n^d }  \widehat{K'_n}(w)\overline{\widehat{K'_n}(w)} \\
&= {n^{-d}} \sum_{w\in \Z_n^d }  K'_n(w)\overline{K'_n(w)}  
 \leq {||K_n||^2_{L^{\infty}(\T^d)}} \leq C n^{-2}.
\end{align*}
This shows immediately that $R_n(u)$ converges in $L^2$ to $0$. 
\end{proof}
We are then done with the proof of (P2) on page \pageref{one_spade}. 
\subsection{Tightness: proof of {(P1)}}\label{subsec:tightness}
We proceed to prove tightness.
Before that, we must introduce a fundamental result: Rellich's theorem.
\begin{theorem}[Rellich's theorem]
If $k_1<k_2$ the inclusion operator $H^{k_2}(\T^d)\hookrightarrow H^{k_1}(\T^d)$ is a compact linear operator. In particular for any radius $R>0$, the closed ball $\overline{B_{\mathcal H_{-\frac{\epsilon}{2}}}(0,\,R)}$ is compact in $\mathcal H_{-\epsilon}$.
\end{theorem}
\begin{proof}[Sketch of the proof] The proof is readily adapted from the one in \citet[Theorem~5.8]{Roe}. Let $\omega>0$ be arbitrarily small. Let $B$ be the unit ball of $H^{k_2}(\T^d)$. We quotient then the space $H^{k_2}(\T^d)$ by the subspace $Z:=\left\{f:\,\widehat f(\nu)=0\text{ for }\|\nu\|>N\right\}$ with $N=N(\omega)$ large enough so that $\|f\|_{{k_1}}<\omega$ for $f\in B\cap Z$. The unitary ball in $H^{k_2}/Z$ is then compact and thus can be covered by finitely many $\omega$-balls, giving a finite $2\omega$-covering of balls for $B$ in the $H^{k_1}$-norm as well. This shows the inclusion operator is compact.

We take $k_1:=-\eps$ and $k_2:=-\frac{\eps}{2}$. By the definitions in Subsection~\ref{subsec:review}, there is a Hilbert space isomorphism between $H^{a}(\T^d)$ and $\mathcal H_{a}(\T^d)$. Applying the above observation, we get the result.
\end{proof}
\begin{proof}[Proof of tightness]
Choose $-\epsilon<-\frac{d}{2}$. Observe that
\[
\|\Xi_n\|_{L^2(\T^d)}^2=16\pi^4  n^{d-4}\sum_{x,\,y\in \T_n^d}\left(\chi_{nx}-\min_{w\in \Z_n^d}\chi_w\right)\left(\chi_{ny}-\min_{w\in \Z_n^d}\chi_w\right)
\]
is a.~s. finite, for fixed $n$, being a finite combination of Gaussian variables and their minimum. Hence $\Xi_n\in L^2(\T^d)\subset \mathcal H_{-\epsilon}(\T^d)$ a.~s.
By Rellich's theorem it will suffice to find, for all $\delta>0$, a $R=R(\delta)>0$ such that
\[
 \sup_{n\in \N}\prob\left(\|\Xi_n\|_{\mathcal H_{-\frac{\epsilon}{2}}}\ge R\right)\le \delta.
 \]
A consequence of Markov's inequality is that such an $R(\delta)$ can be found as long as we show that for some $C>0$
\[
 \sup_{n\in \N}\E\left[\|\Xi_n\|_{\mathcal H_{-\frac{\epsilon}{2}}}^2\right]\le C.
\]
Since $\Xi_n\in L^2$, it admits a Fourier series representation $\Xi_n(\vartheta)=\sum_{\nu\in\Z^d}\widehat{\Xi_n}(\nu)\mathbf e_\nu(\vartheta)$
with $\widehat{\Xi_n}(\nu)=(\Xi_n,\,\mathbf e_\nu)_{L^2(\T^d)}$.
Thus we can express
\begin{align*}
\|\Xi_n\|_{\mathcal H_{-\frac{\epsilon}{2}}}^2&=\sum_{\nu\in \Z^d\setminus\{0\}}\|\nu\|^{-2\epsilon}\left|\widehat{\Xi_n}(\nu)\right|^2.
\end{align*}
Observe that
\[
\widehat{\Xi_n}(\nu)=\int_{\T^d}\Xi_n(\vartheta)\mathbf e_{\nu}(\vartheta)\De\vartheta=4\pi^2\sum_{x\in \T^d_n}n^{\frac{d-4}{2}}\chi_{nx}\int_{B(x,\,\frac{1}{2n})}\mathbf e_{\nu}(\vartheta)\De\vartheta.
\]
This gives
\begin{align}
\E\left[\|\Xi_n\|_{\mathcal H_{-\frac{\epsilon}{2}}}^2\right]&=16\pi^4\sum_{\nu\in \Z^d\setminus\{0\}}\sum_{x,\,y\in \T^d_n}\|\nu\|^{-2\epsilon}n^{d-4}\E\left[\chi_{nx}\chi_{ny}\right]\int_{B(x,\,\frac{1}{2n})}\mathbf e_{\nu}(\vartheta)\De\vartheta\int_{B(y,\,\frac{1}{2n})}\overline{\mathbf e_{\nu}(\vartheta)}\De\vartheta\nonumber\\
&=16\pi^4\sum_{\nu\in \Z^d\setminus\{0\}}\sum_{x,\,y\in \T^d_n}\|\nu\|^{-2\epsilon}n^{d-4}H(nx,\,ny)\int_{B(x,\,\frac{1}{2n})}\mathbf e_{\nu}(\vartheta)\De\vartheta\int_{B(y,\,\frac{1}{2n})}\overline{\mathbf e_{\nu}(\vartheta)}\De\vartheta.\label{eq:wehave}
\end{align}
Let us denote by $F_{n,\,\nu}:\T^d_n\to \R$ the function $F_{n,\,\nu}(x):=\int_{B(x,\,\frac{1}{2n})}\mathbf e_{\nu}(\vartheta)\De\vartheta$. Since $\mathbf e_\nu\in L^2(\T^d)$, the Cauchy-Schwarz inequality implies that $F_{n,\nu}\in L^1(\T^d)$. 

Assume we can prove
\begin{claim}\label{claim:song}
There exists $C'>0$ such that 
\begin{equation}\label{eq:first-aid}
\sup_{\nu\in \Z^d}\sup_{n\in \N}\sum_{x,\, y\in \T^d_n}n^{d-4}H(nx,\,ny)F_{n,\,\nu}(x)\overline{F_{n,\,\nu}(y)}\le C'.
\end{equation}
\end{claim}
Using the above Claim and $-\eps<-\frac{d}{2},$ from \eqref{eq:wehave} we get
\begin{align*}
\E\left[\|\Xi_n\|_{\mathcal H_{-\frac{\epsilon}{2}}}^2\right]=
&16\pi^4\sum_{\nu\in \Z^d\setminus\{0\}}\|\nu\|^{-2\epsilon}\sum_{x, \, y\in \T^d_n}n^{d-4}H(nx,\,ny)F_{n,\,\nu}(x)\overline{F_{n,\,\nu}(y)}\\
&\le C'\sum_{k\ge 1}k^{d-1-2\epsilon}\le C.
\end{align*}
This concludes the proof, assuming Claim~\ref{claim:song}. 
\end{proof}
We are then left to show the claim we have made:
\begin{proof}[Proof of Claim~\ref{claim:song}]
First we use the bound \eqref{eq:mongoose} and the fact that $$\sum_{x,\,y\in \T_n^d} \exp(2\pi \i (x-y)\cdot w) F_{n,\,\nu}(x) \overline{F_{n,\,\nu}(y)}=\left|\widehat{F_{n,\,\nu}}(w)\right|^2 n^{2d}\ge 0$$ to obtain
\begin{align}
&\sum_{x,\,y\in \T^d_n}n^{d-4}H(nx,\,ny)F_{n,\,\nu}(x)\overline{F_{n,\,\nu}(y)}\nonumber\\
&=\sum_{x,\,y\in \T^d_n}\frac{n^{d-4}n^{-d}}{16}\sum_{w\in \Z_n^d\setminus \{0\}}\frac{\exp(2\pi \i (x-y)\cdot w)}{\left(\sum_{i=1}^d\sin^2\left(\pi\frac{w_i}{n}\right)\right)^2}F_{n,\,\nu}(x)\overline{F_{n,\,\nu}(y)}\nonumber\\
&\overset{\eqref{eq:mongoose}}\le C \sum_{x,\,y\in \T^d_n}\sum_{w\in \Z_n^d\setminus \{0\}}\frac{\exp(2\pi \i (x-y)\cdot w)}{\|w\|^4}F_{n,\,\nu}(x)\overline{F_{n,\,\nu}(y)}\label{eq:back}
\end{align}
Choose a mollifier $\phi_\kappa$ as in the previous considerations (see below \eqref{eq:zeta}). We rewrite the expression in the right-hand side of~\eqref{eq:back} accordingly as 
\begin{align}
&C\sum_{x,\,y\in \T^d_n}\sum_{w\in \Z_n^d\setminus \{0\}}\widehat{\phi_\kappa}(w)\frac{\exp(2\pi \i (x-y)\cdot w)}{\|w\|^4}F_{n,\,\nu}(x)\overline{F_{n,\,\nu}(y)}\nonumber\\
&+C\sum_{x,\,y\in \T^d_n}\sum_{w\in \Z_n^d\setminus \{0\}}\left(1-\widehat{\phi_\kappa}(w)\right)\frac{\exp(2\pi \i (x-y)\cdot w)}{\|w\|^4}F_{n,\,\nu}(x)\overline{F_{n,\,\nu}(y)}.\label{eq:moghli}
\end{align}
First we get a bound for the second term. Denote as $G_{n,\,\nu}:\,\Z_n^d\to \R$ the rescaled function $G_{n,\,\nu}(z):= F_{n,\,\nu}(\frac{z}{n})$. Now we have
\begin{align*}
&C\sum_{x,\,y\in \T^d_n}\sum_{w\in \Z_n^d\setminus \{0\}}\left(1-\widehat{\phi_\kappa}(w)\right)\frac{\exp(2\pi \i (x-y)\cdot w)}{\|w\|^4}F_{n,\,\nu}(x)\overline{F_{n,\,\nu}(y)}\\
&=C\sum_{w\in \Z_n^d\setminus \{0\}}\frac{1-\widehat{\phi_\kappa}(w)}{\|w\|^4}\sum_{x,\,y\in \Z^d_n}F_{n,\,\nu}(\frac{x}{n})\overline{F_{n,\,\nu}(\frac{y}{n})}\exp\left(2\pi \i (x-y)\cdot \frac{w}{n}\right)\\
&=C n^{2d}\sum_{w\in \Z_n^d\setminus \{0\}}\frac{1-\widehat{\phi_\kappa}(w)}{\|w\|^4} \widehat{G_{n,\,\nu}}(w) \overline{\widehat{G_{n,\,\nu}}(w)}\overset{\eqref{eq:banjaara_two}}\le C \kappa n^{2d} \sum_{w\in \Z_n^d}\left|\widehat{G_{n,\,\nu}}(w)\right|^2
\end{align*}
where in the last inequality we have used that $\|w\|\ge 1$ and $\left|\widehat{G_{n,\,\nu}}(0)\right|^2\ge 0$.
The description of $G_{n,\,\nu}$, the fact that $|F_{n,\,\nu}(w)|\le n^{-d}$ and Parseval give 
\begin{align}
 \sum_{w\in \Z_n^d}\left|\widehat{G_{n,\,\nu}}(w)\right|^2&= n^{-d} \sum_{w\in \Z_n^d} G_{n,\,\nu}(w)\overline{G_{n,\,\nu}(w)}=n^{-d}\sum_{w\in \T_n^d} F_{n,\,\nu}(w) \overline{F_{n,\,\nu}(w)}\nonumber\\
 &\le n^{-2d}\sum_{w\in \T_n^d}\int_{B(w,\,\frac1{2n})} |\mathbf e_\nu(\vartheta)|\De \vartheta=n^{-2d}\int_{\T^d} |\mathbf e_\nu(\vartheta)|\De \vartheta \nonumber\\
 &\le n^{-2d}\|\mathbf e_\nu\|_{L^1(\T^d)}\le Cn^{-2d}.\label{eq:firstbreadth}
\end{align}
By means of \eqref{eq:firstbreadth} we get that 
\eq{}\label{eq:secondbreadth}C\sum_{x,\,y\in \T^d_n}\sum_{w\in \Z_n^d\setminus \{0\}}\left(1-\widehat{\phi_\kappa}(w)\right)\frac{\exp(2\pi \i (x-y)\cdot w)}{\|w\|^4}F_{n,\,\nu}(x)\overline{F_{n,\,\nu}(y)}\le C\kappa.\eeq{}
We are back to bounding the first term in~\eqref{eq:moghli}. 
\begin{align*}
&C\sum_{x,\,y\in \T^d_n}\sum_{w\in \Z_n^d\setminus \{0\}}\widehat{\phi_\kappa}(w)\frac{\exp(2\pi \i (x-y)\cdot w)}{\|w\|^4}F_{n,\,\nu}(x)\overline{F_{n,\,\nu}(y)}\\
&=C\sum_{x,\,y\in \T^d_n}\sum_{w\in \Z^d\setminus\{0\}}\widehat{\phi_\kappa}(w)\frac{\exp(2\pi \i (x-y)\cdot w)}{\|w\|^4}F_{n,\,\nu}(x)\overline{F_{n,\,\nu}(y)}\\
&\qquad-C\sum_{x,\,y\in \T^d_n}\sum_{w\in \Z^d:\,\|w\|_\infty>n}\widehat{\phi_\kappa}(w)\frac{\exp(2\pi \i (x-y)\cdot w)}{\|w\|^4}F_{n,\,\nu}(x)\overline{F_{n,\,\nu}(y)}.
\end{align*}
Using \eqref{eq:decay_phi} we obtain a bound on the second term as
\begin{align}
 \sum_{x,\,y\in \T^d_n}&\sum_{w\in \Z^d:\,\|w\|_\infty>n}\widehat{\phi_\kappa}(w)\frac{\exp(2\pi \i (x-y)\cdot w)}{\|w\|^4}F_{n,\,\nu}(x)\overline{F_{n,\,\nu}(y)}\nonumber\\
&\le C\sum_{x,\,y\in \T^d_n}\sum_{w\in \Z^d:\,\|w\|_\infty>n}n^{-4}\left|\widehat{\phi_\kappa}(w)\right|\left|F_{n,\,\nu}(x)\overline{F_{n,\,\nu}(y)}\right|\nonumber\\
&\le C\sum_{w\in \Z^d:\,\|w\|_\infty>n} \left|\widehat{\phi_\kappa}(w)\right| \left(\sum_{x\in \T_n^d} |F_{n,\,\nu}(x)|\right)^2\le C\sum_{w\in \Z^d:\,\|w\|_\infty>n} \frac{\|\mathbf e_\nu\|_{L^1(\T^d)}^2}{(1+\|w\|)^{\delta}}\le C.\label{eq:amar}
\end{align}
Finally \eqref{eq:decay_phi} tells us that
\begin{align}
&\sum_{x,\,y\in \T^d_n}\sum_{w\in \Z^d\setminus\{0\}}\widehat{\phi_\kappa}(w)\frac{\exp(2\pi \i (x-y)\cdot w)}{\|w\|^4}F_{n,\,\nu}(x)\overline{F_{n,\,\nu}(y)}\nonumber\\
&\le C\sum_{x,\,y\in \T_n^d} \sum_{w\in \Z^d}\frac{1}{(1+\|w\|)^{\delta}}\left|F_{n,\,\nu}(x)\overline{F_{n,\,\nu}(y)}\right|\le C\sum_{w\in \Z^d} \frac1{\left(1+\|w\|\right)^{\delta}} \|\mathbf e_\nu\|_{L^1(\T^d)}^2\le C, \label{eq:bhaier}
\end{align}
where $C$ possibly depends on $\kappa$ and $\delta$.
Plugging in \eqref{eq:first-aid} the expressions \eqref{eq:secondbreadth}, \eqref{eq:amar} and \eqref{eq:bhaier} we can draw the required conclusion.
\end{proof}
This gives a proof of (P1) on page \pageref{one_spade} and completes the proof of Theorem~\ref{thm:main}.
\section{Proof of Theorem~\ref{thm:3}}\label{sec:thm3}
\paragraph{\it Strategy of the proof.}We will argue as in Theorem~\ref{thm:main} and need thus to show both (P1) and (P2). While (P2) will follow almost in the same way as in the Gaussian case, (P1) will require a different approach. Firstly, we will need to remove constants in defining $e_n$ so that we will end up working with a field depending only on linear combinations of $(\sigma(x))_{x\in \Z_n^d}$. Secondly, we will show in Subsection~\ref{subsec:bdd} that, for $\sigma$ bounded a.~s., the convergence to the bilaplacian field is ensured via the moment method. Lastly, we will truncate the weights $\sigma$ at a level $\cR>0$ and show that the truncated field approximates the original one.
\paragraph{\it Reduction to a bounded field}We first recall some facts from \cite{LMPU}.  Note that odometer $e_n$ satisfies
\begin{equation*}
\begin{cases}
\Delta_g e_n(x)= 1-s(x),  \\
 \min_{z\in \Z_n^d} e_n(z)=0.
 \end{cases}
\end{equation*}
Also if one defines 
\begin{equation}\label{eq:vgreen}
v_n(y)=\frac1{2d} \sum_{x\in \Z_n^d} g(x,y) (s(x)-1),
\end{equation}
then $\Delta_g(e_n-v_n)(z)=0$. Since any harmonic function on a finite connected graph is constant, it follows from the proof of Proposition 1.3 of \cite{LMPU} that the odometer has the following representation also in the case where the weights are non-Gaussian:
\begin{equation}
e_n(x)= v_n(x)-\min_{z\in \Z_n^d}v_n(z).
\end{equation}
Let us define the following functional: for any function $h_n: \Z_n^d\to \R$ set
$$\Xi_{h_n}(x):=4\pi^2\sum_{z\in \T^d_n}n^{\frac{d-4}{2}}h_n({nz})\one_{B\left(z,\,\frac{1}{2n}\right)}(x),\quad x\in \T^d.$$
Note that for $u\in C^\infty(\T^d)$ such that $\int_{\T^d} u(x)\De x=0$ it follows immediately that
$$\la \Xi_{e_n}, u\ra = \la \Xi_{v_n}, u\ra.$$
Observe that 
$$s(x)-1= \sigma(x)-\frac1{n^d} \sum_{y\in \Z_n^d} \sigma(y)$$ 
and hence we have from~\eqref{eq:vgreen}
$$v_n(y)=\frac1{2d}\sum_{x\in \Z_n^d} g(x,y)\sigma(x)-\frac1{2d n^{d}} \sum_{x\in \Z_n^d} g(x,y) \sum_{z\in \Z_n^d} \sigma(z).$$
By~\eqref{eq:RTL} it follows that $(2d)^{-1} \sum_{x\in \Z_n^d} g(x,y)= (2d)^{-1}n^{-d}\sum_{w\in \Z_n^d} \E_y[\tau_w]$ which is independent of $y$. We can then say that
$$v_n(y)= \frac1{2d}\sum_{x\in \Z_n^d} g(x,y)\sigma(x)- Cn^{-d}\sum_{z\in\Z_n^d} \sigma(z).$$
If we call 
\[w_n(y):=\left(2d\right)^{-1}\sum_{x\in \Z_n^d} g(x,y)\sigma(x),\]
by the mean-zero property of the test functions it follows that $\la \Xi_{v_n}, u\ra = \la \Xi_{w_n}, u \ra.$ Therefore we shall reduce ourselves to study the convergence of the field $\Xi_{w_n}.$ To determine its limit, we will first prove that all moments of $\Xi_{w_n}$ converge to those of $\Xi$; via characteristic functions, we will show that the limit is uniquely determined by moments.
\subsection{Scaling limit with bounded weights}\label{subsec:bdd}
The goal of this Subsection is to determine the scaling limit for bounded weights, namely to prove
\begin{theorem}[Scaling limit for bounded weights]
Assume $(\sigma(x))_{x\in \Z_n^d}$ is a collection of i.i.d. variables with $\E\left[\sigma\right]=0$ and $\E\left[\sigma^2\right]=1$. Moreover assume there exists $K<+\infty$ such that $|\sigma|\le K$ almost surely. Let $d\ge 1$ and $e_n(\cdot)$ be the corresponding odometer. Then if we define the formal field $\Xi_n$ as in \eqref{eq:deffield} for such i.i.d. weights, then it
converges in law as $n\to+\infty$ to the bilaplacian field $\Xi$ on $\T^d$. The convergence holds in the same fashion of Theorem~\ref{thm:main}.
\end{theorem}
Before showing this result, we must prove an auxiliary Lemma.
It gives us a uniform estimate in $n$ on the Fourier series of the mean of $u$ in a small ball.
\begin{lemma}\label{lemma:espresso}
Fix $u\in C^\infty(\T^d)$ with zero average. If we define
\begin{align*}
T_n:\T^d&\to \R \\
z&\mapsto \int_{B(z,\,\frac{1}{2n})}u(y)\De y
\end{align*}
and $\mathcal T_n:\Z_n^d\to\R$ is defined as $\mathcal T_n(z):=T_n\left(\frac{z}{n}\right),$ then for $n$ large enough we can find a constant $\mathcal M:=\mathcal M(d,\,u)<+\infty$ such that
$$n^{d}\sum_{z\in \Z_n^d} \left|\widehat{ \mathcal T_n}(z)\right|\le \mathcal M.$$ 
\end{lemma}

\begin{proof}
For $z\in \Z_n^d$ we can write
\begin{align}
\widehat{\mathcal T_n}(z)& = \la \mathcal T_n, \,\psi_z\ra = \frac1{n^d} \sum_{y\in \Z_n^d} \mathcal T_n(y) \psi_{-z}(y)\nonumber\\
&= \frac1{n^d} \sum_{y\in\Z_n^d}  T_n\left(\frac{y}{n}\right) \exp\left(-2\pi\i z\cdot \frac{y}{n}\right)=\frac1{n^d} \sum_{y\in\T_n^d} T_n(y) \exp(-2\pi\i z\cdot y).\label{eq:octo}
\end{align}
Since $u\in C^\infty(\T^d)$, one can take derive under the integral sign and get that $T_n\in C^\infty(\T^d)$, so $\sum_{z\in \Z^d} \left|\widehat{T_n}(z)\right|<+\infty$. Hence by the Fourier inversion theorem we have the following inversion formula to be valid for every $y\in \T^d$:
$$T_n(y)=\sum_{w\in \Z^d} \widehat{T_n}( w) \exp\left( 2\pi \i y\cdot w\right).$$
First we split the sum above according to the norm of $w$ and plug it in \eqref{eq:octo}. Namely we get
\begin{align}
\widehat{\mathcal T_n}(z)&= \frac1{n^d} \sum_{y\in\T_n^d} T_n(y) \exp(-2\pi\i z\cdot y)= \frac1{n^d} \sum_{y\in \T_n^d} \sum_{ w\in \Z_n^d} \widehat{T_n}(w) \exp(2\pi\i w\cdot y)\exp(-2\pi\i z\cdot y)\nonumber\\
&+\frac1{n^d} \sum_{y\in \T_n^d} \sum_{ w\in \Z^d:\,\|w\|_\infty>n} \widehat{T_n}(w) \exp(2\pi\i w\cdot y)\exp(-2\pi\i z\cdot y).\label{eq:alo_ozi}
\end{align}
Let us look at the first summation: using the orthogonality of the characters of $L^2(\Z^d_n)$ we can write
\begin{align*}
\frac1{n^d} \sum_{y\in \T_n^d} \sum_{ w\in \Z_n^d}& \widehat{T_n}(w) \exp(2\pi\i w\cdot y)\exp(-2\pi\i z\cdot y)\\
&= \frac1{n^d}  \sum_{ w\in \Z_n^d} \widehat{T_n}(w) \sum_{y\in \Z_n^d}\exp\left(2\pi\i w\cdot \frac{y}{n}\right)\exp\left(-2\pi\i z\cdot \frac{y}{n}\right)\\
&= \frac1{n^d} \sum_{w\in \Z_n^d} \widehat T_n(w) n^d\one_{w=z}= \widehat{T_n}(z).
\end{align*}
Noting that
\[
\widehat{\mathcal T_n}(0) =\frac1{n^d} \sum_{y\in \T_n^d}  T_n(y)=\frac1{n^d}\sum_{y\in \T_n^d} \int_{B\left(y, \frac1{2n}\right)}u(x)\De x=\frac1{n^d}\int_{\T^d} u(x)\De x=0,
\]
this means we need to show that $\sum_{z\in \Z_n^d\setminus \{0\}} \left|\widehat{T_n}(z)\right| \le C(d)n^{-d}$. We follow the proof of \citet[Corollary 1.9, Chapter VII]{stein:weiss}.   For a multi-index $\alpha=(\alpha_1,\,\ldots,\,\alpha_d)\in \N^d$ and a point $x=(x_1,\,\ldots,\,x_d)\in \R^d$ we set 
\[
x^\alpha:=\prod_{j=1}^d x_j^{\alpha_j}
\]
and adopt the convention $0^0=1$. We choose now a smoothness parameter $k_0>d$. For any $\alpha$ with $|\alpha|:=\alpha_1+\cdots +\alpha_d\le k_0$ we can find a constant $c=c(k_0,\,d)$ such that 
$$\sum_{\alpha:\,|\alpha|=k_0}4\pi^2z^{2\alpha}\ge c\|z\|^{2k_0}.$$
Note that
\begin{align*}
&\sum_{z\in \Z_n^d\setminus \{0\}} \left|\widehat{T_n} (z)\right|\le \sum_{z\in \Z_n^d\setminus \{0\}} \left|\widehat{T_n} (z)\right|\left( \sum_{\alpha:\,|\alpha|=k_0} 4\pi^2 z^{2\alpha}\right)^{\frac{1}{2}}\|z\|^{-k_0}c^{-\frac{1}{2}}\\
&\le \left( \sum_{z\in \Z_n^d\setminus \{0\}} \left|\widehat{T_n} (z)\right|^2 \sum_{\alpha:\,|\alpha|=k_0} 4\pi^2 z^{2\alpha}\right)^{\frac{1}{2}} \left( \sum_{z\in \Z_n^d\setminus \{0\}} \|z\|^{-2k_0}\right)^{\frac{1}{2}}c^{-\frac{1}{2}}.
\end{align*}
Here we have used the Cauchy-Schwarz inequality in the last step. Now since $\sum_{z\in \Z_n^d\setminus \{0\}} \|z\|^{-2k_0}<+\infty$ we can compute a constant $C$ such that
\eq{}\sum_{z\in \Z_n^d\setminus \{0\}} \left|\widehat{T_n} (z)\right| \le C \left( \sum_{z\in \Z_n^d\setminus \{0\}} \left|\widehat{T_n} (z)\right|^2 \sum_{\alpha:\,|\alpha|=k_0} 4\pi^2 z^{2\alpha}\right)^{\frac{1}{2}}\le C\left( \sum_{\alpha:\,|\alpha|=k_0} \sum_{z\in \Z^d} \left|\widehat{T_n} (z)\right|^2 4\pi^2 z^{2\alpha}\right)^{\frac{1}{2}}.\label{eq:bound_hat_Tn}\eeq{}
Let us call $D^\alpha$ the derivative with respect to $\alpha$. Using the rule of derivation of Fourier transforms \cite[Chapter I, Theorem~1.8]{stein:weiss} and Parseval we have that
$$\sum_{z\in \Z^d} \left|\widehat{T_n} (z)\right|^2 4\pi^2 z^{2\alpha}= \int_{\T^d} \left|D^\alpha T_n(x)\right|^2\De x.$$
By the smoothness of $u$ we deduce that
\begin{equation}\label{eq:pond}
|D^\alpha T_n(x)| \le \|D^\alpha u\|_{L^\infty(\T^d)} \int_{B(0,\frac{1}{2n})} \De w= \|D^\alpha u\|_{L^\infty(\T^d)} (2n)^{-d}.
\end{equation}
Plugging this estimate in \eqref{eq:bound_hat_Tn} we get that 
$$\sum_{z\in \Z_n^d\setminus\{0\}} \left|\widehat{T_n} (z)\right|^2 \le C n^{-{d}} \left(\sum_{\alpha:\,|\alpha|=k_0}\|D^\alpha u\|_{L^\infty(\T^d)}^2 \right)^{\frac{1}{2}}.$$
This finally gives that 
$$\sum_{z\in \Z_n^d\setminus \{0\}} \left|\widehat{T_n} (z)\right| \le C(k_0,\,d,\,u) n^{-d}. $$
For the second summand of \eqref{eq:alo_ozi} observe that
\[
\int_{\T^d}{D^\alpha T_n}(w)\e^{-2\pi \i z\cdot w}\De w=(2\pi\i z)^\alpha\widehat{T_n}(z),\quad\alpha\in \N^d.
\]
The parameter $\alpha$ will be chosen later so that the second summand is of lower order than the first. By \eqref{eq:alo_ozi} and \eqref{eq:pond}
\[
\left|\widehat{T_n}(z)\right|\le \frac{2^{-d-1}\|D^\alpha u\|_{L^\infty(\T^d)}}{ \pi n^d\left|z^\alpha\right|}.
\]
We use this estimate to get 
\begin{align*}
\frac1{n^d} \sum_{y\in \T_n^d} \sum_{\|w\|_\infty> n}& \widehat{T_n}(w) \exp(2\pi\i w\cdot y)\exp(-2\pi\i z\cdot y)\le  \sum_{\|w\|_\infty> n} \left|\widehat{T_n}(w)\right| \\
&\le \frac{C(u,\,d,\,\alpha)}{  n^d}\sum_{\ell=n}^{+\infty}\frac{\ell^{d-1}}{\ell^{|\alpha|}}\le C(u,\,d,\,\alpha)n^{-|\alpha|}\left(1+\O{n^{-1}}\right).
\end{align*}
Thus choosing $\alpha $ with $|\alpha|>d$ we find a constant $\mathcal M=\mathcal M(d,\,u)$ such that
$$\sum_{z\in \Z_n^d} \left|\widehat{\mathcal T_n}(z)\right|\le \mathcal M n^{-d}$$
as we wanted to show.
\end{proof}
We can now start with the moment method, and we being with moment convergence.
\paragraph{\it Moment convergence}We now show that all moments converge to those of the required limiting distribution. This is explained in the following Proposition.
\begin{proposition}\label{prop:moments}
Assume $\E\left[\sigma\right]=0$, $\E\left[\sigma^2\right]=1$ and that there exists $K<+\infty$ such that $|\sigma|\le K$ almost surely. Then for all $m\ge 1$ and all $u\in C^\infty(\T^d)$ with zero average, the following limits hold:
\eq{}\label{eq:case_mom}
\lim_{n\to+\infty}\E\left[ \la \Xi_{w_n}, u\ra^{m}\right] =\begin{cases} (2m-1)!! \|u\|_{-1}^{m},&m\in 2\N\\
 0,&m\in 2\N+1.
\end{cases}
\eeq{}
\end{proposition}
\begin{proof}
We will first show that the $m=2$ case satisfies the claim.
\paragraph{\it Case \texorpdfstring{$m=2$}{aa}.}
We have the equality
\begin{align*}
\E\left[ w_n(y)w_n(y')\right]&=(2d)^{-2} \sum_{x\in \Z_n^d} g(x,y)\sum_{x'\in \Z_n^d} g(x',y') \E[\sigma(x)\sigma(x')].
\end{align*}
The independence of the weights gives
\[
\E\left[ \la \Xi_{w_n}, u\ra^2\right] =16\pi^4 \frac{n^{d-4}}{4 d^2} \sum_{x\in \Z_n^d}\left(\sum_{z\in \T_n^d}g(x,\,nz)T_n(z)\right)^2.
\]
With the same argument of the proof of Proposition~\ref{prop:decomposition} one has
\eq{}\label{eq:gio}
(2d)^{-2} \sum_{x\in \Z_n^d} g(x,y)g(x,y')=  n^d L^2+ H(y,y')
\eeq{}
so that, using that test functions have zero average,
\begin{align*}
\E\left[ \la \Xi_{w_n}, u\ra^2\right] &=16\pi^4 \frac{n^{d-4}}{4 d^2} \sum_{x\in \Z_n^d}\left(\sum_{z\in \T_n^d}g(x,\,nz)T_n(z)\right)^2\\
&=16\pi^4 {n^{d-4}}  \sum_{z,\,z'\in \T_n^d} H(nz, nz')T_n(z)T_n(z')\\
&=16\pi^4 {n^{d-4}}  \sum_{z,\,z'\in \T_n^d} H(nz, nz') \int_{B(z,\,\frac{1}{2n})} u(x)\De x\int_{B(z',\,\frac{1}{2n})} u(x')\De x'.
\end{align*}
Now we break the above sum into the following 3 sums (recall $K_n(u)$ from~\eqref{eq:K_n}):
\begin{align*}
\E\left[ \la \Xi_{w_n}, u\ra^2\right]&= 16\pi^4 n^{d-4} \sum_{z,\,z'\in \T_n^d}n^{-2d} H(nz, nz')  u(z) u(z')\\
&+16\pi^4 n^{d-4} \sum_{z,\,z'\in \T_n^d}n^{-2d} H(nz, nz')  K_n(z) K_n(z')\\
&+32\pi^4  n^{d-4} \sum_{z,\,z'\in \T_n^d}n^{-2d} H(nz, nz')  K_n(z) u(z').
\end{align*}
A combination of Proposition~\ref{lem:bhindi} and Proposition~\ref{lem:small} with the Cauchy-Schwarz inequality shows that the first term converges to $\|u\|^2_{-1}$ in the limit $n\to+\infty$ and the other two go to zero.

Having concluded the case $m=2$, we would like to see what the higher moments look like. Let us take for example $m=3$, in which case
\begin{align*}
&\E\left[ \la \Xi_{w_n}, u\ra^3\right]= \left(\frac{4\pi ^2 n^{\frac{d-4}{2}}}{2d}\right)^3 \sum_{z_1, \,z_2, \,z_3\in \T_n^d} \E\left[ w(nz_1)w(nz_2)w(nz_3)\right] T_n(z_1) T_n(z_2)T_n(z_3)\\
&=\left(\frac{2\pi ^2 n^{\frac{d-4}{2}}}{d}\right)^3 \sum_{z_1, \,z_2, \,z_3\in \T_n^d} \sum_{x_1,\,x_2,\,x_3\in \Z_n^d}\E\left[ \prod_{j=1}^3\sigma(x_j) \right]\prod_{j=1}^3g(x_j,\,nz_j)T_n(z_j)\\
& = \left(\frac{2\pi ^2 n^{\frac{d-4}{2}}}{d}\right)^3 \sum_{z_1, \,z_2, \,z_3\in \T_n^d} \sum_{x \in \Z_n^d}\E\left[ \sigma^3(x) \right]\prod_{j=1}^3g(x,\,nz_j)T_n(z_j)\\
& = \left(\frac{2\pi ^2 n^{\frac{d-4}{2}}}{d}\right)^3 \E\left[ \sigma^3 \right]\sum_{x \in \Z_n^d} \left [\sum_{z \in \T_n^d} g(x,\,nz)T_n(z)\right ]^3.
\end{align*}
More generally, let us call $\mathscr P(n)$ the set of partitions of $\{1,\,\ldots,\,n\}$ and as $\mathscr P_2(n)\subset \mathscr P(n)$ the set of pair partitions. We denote as $\Pi$ a generic block of a partition $P$ and as $|\Pi|$ its cardinality (for example, $\Pi=\{1,\,2,\,3\}$ is a block of cardinality $3$ of $P=\{\{1,\,2,\,3\},\,\{4\}\}\in \mathscr P(4)$).
Observe that
\begin{align}
&\E\left[ \la \Xi_{w_n}, u\ra^{m}\right]=\left(\frac{2\pi ^2 n^{\frac{d-4}{2}}}{d}\right)^m\sum_{z_1,\,\ldots,\,z_m\in \T_n^d}\E\left[\prod_{j=1}^m w_n(nz_j)\right]\prod_{j=1}^m T_n(z_j)\nonumber\\
&=\left(\frac{2\pi ^2 n^{\frac{d-4}{2}}}{d}\right)^m\sum_{P\in\mathscr P(m)}\prod_{\Pi\in P}\E\left[\sigma^{|\Pi|}\right] \sum_{x \in \Z_n^d }\left ( \sum_{ \substack{   z_j \in \mathbb{T}^d_n:\,j \in \Pi}} \,\prod_{j\in \Pi} g(x,\, nz_j)T_n(z_j)\right ) \nonumber\\
&=\sum_{P\in\mathscr P(m)}\prod_{\Pi\in P}\left(\frac{2\pi ^2 n^{\frac{d-4}{2}}}{d}\right)^{|\Pi|}\E\left[\sigma^{|\Pi|}\right]\sum_{x \in \Z_n^d } \left ( \sum_{z \in \mathbb{T}^d_n} g(x,\,nz)T_n(z) \right )^{|\Pi|}\label{eq:moments}.
\end{align}
For a fixed $P$, let us consider in the product over $\Pi\in P$ any term corresponding to a block $\Pi$ with $|\Pi|=1$: this will give no contribution because $\sigma $ is centered. Consider instead $\Pi\in P$ with $\ell:=|\Pi|>2.$ We see that
\begin{align*}
\left(\frac{2\pi ^2 n^{\frac{d-4}{2}}}{d}\right)^\ell&\E
\left[\sigma^\ell\right]\sum_{x\in \Z_n^d} \left ( \sum_{z \in \mathbb{T}^d_n} g(x,\,nz)T_n(z) \right )^{l} \nonumber\\
&=\left(\frac{2\pi ^2 n^{\frac{d-4}{2}}}{d}\right)^\ell\E
\left[\sigma^\ell\right]\sum_{x\in \Z_n^d}\left(\sum_{z\in \Z_n^d}g(x,\,z){\mathcal T_n}(z)\right)^\ell.
\end{align*}
Applying Parseval the above expression equals
\begin{align}
\left(\frac{2\pi^2 n^{\frac{d-4}{2}}}{d}\right)^\ell&\E\left[\sigma^\ell\right]\sum_{x\in \Z_n^d}\left(n^{d}\sum_{z\in \Z_n^d}\widehat{g_x}(z) \widehat{\mathcal T_n}(z)\right)^\ell\nonumber\\
&\stackrel{\eqref{eq:20}}{=}\left({4\pi^2 n^{\frac{d-4}{2}}}\right)^\ell\E\left[\sigma^\ell\right]\sum_{x\in \Z_n^d}\left(\sum_{z\in \Z_n^d\setminus\{0\}}\frac{\psi_{-z}(x)}{-\lambda_z} \widehat{\mathcal T_n}(z)\right)^\ell\label{eq:badBeard}.
\end{align}
Here we have used that $\widehat{\mathcal T_n}(0)=0.$ Thanks to the fact that $-\lambda_z\ge Cn^{-2}$ uniformly over $z\in\Z_n^d\setminus \{0\}$ (see \eqref{eq:mongoose}) we obtain
\begin{align}
&\left(\frac{2\pi^2 n^{\frac{d-4}{2}}}{d}\right)^\ell\E
\left[\sigma^\ell\right]\sum_{x\in \Z_n^d}\left ( \sum_{z \in \mathbb{T}^d_n} g(x,\,nz)T_n(z) \right )^{l}\le C\E\left[\sigma^\ell\right]  n^{\frac{\ell d}{2}+d} \left(\sum_{z\in \Z_n^d\setminus\{0\}} \left|\widehat{\mathcal T_n}(z)\right|\right)^\ell.\label{eq:bound_ell}
\end{align}
Since $\sigma$ is almost surely bounded, by Lemma~\ref{lemma:espresso} we can conclude that each term in \eqref{eq:moments} corresponding to a block of cardinality $\ell>2$ has order at most $n^{\frac{ \ell d}{2}-(\ell-1)d}=\o{1}$. Hence in \eqref{eq:moments} only pair partitions of $m$ will give a contribution of order unity to the sum. Since, for $m:=2m'+1$, there are no pair partitions, $\E\left[ \la \Xi_{w_n}, u\ra^{2m'+1}\right]$ will converge to zero. Otherwise, for $m:=2m'$ we can rewrite
\begin{align*}
\E\left[ \la \Xi_{w_n}, u\ra^{2m'}\right]&=\sum_{P\in\mathscr P_2(2m')}\left(\frac{4\pi^4 n^{{d-4}}}{d^2}\sum_{x\in\Z_n^d }\left(\sum_{z\in \Z_n^d}g(x,\,z)\mathcal T_n(z)\right)^2\right)^{m'}+\o{1}.
\end{align*}
Since $\left|\mathscr P_2(m)\right|=(2m-1)!!$ and the term in the bracket above converges to $\|u\|_{-1}^{2}$ we can conclude the proof of Proposition~\ref{prop:moments}.
\end{proof}

\paragraph{\it Tightness.}The proof of tightness is, not suprisingly, a re-run of that in the Gaussian case. In fact tightness depends on the covariance structure of the field we are examining; since both the Gaussian functional $\Xi_n$ and $w_n$ share the same covariance, we can recover mostly of the results already calculated. First we notice that
\[
\lVert{\Xi_{w_n}}\rVert_{L^2(\T^d)}^2=\frac{16\pi^4}{(2d)^2}  n^{d-4}\sum_{x,\,y\in \Z_n^d}g(x,\,y)\sigma(x)\sum_{x',\,y'\in \Z_n^d}g(x',\,y')\sigma(x')
\]
is finite with probability one, since $\sigma$ is bounded. One can then go along the lines of the proof of (P1) in Subsection~\ref{subsec:tightness} and get to \eqref{eq:wehave} which will become, in our new setting,
\begin{align*}
&\frac{16\pi^4}{(2d)^2}\sum_{\nu\in \Z^d\setminus\{0\}}\sum_{x,\,y\in \T^d_n}\|\nu\|^{-2\epsilon}n^{d-4}\E\left[w_n(nx)w_n(ny)\right]\int_{B(x,\,\frac{1}{2n})}\mathbf e_{\nu}(\vartheta)\De\vartheta\int_{B(y,\,\frac{1}{2n})}\overline{\mathbf e_{\nu}(\vartheta)}\De\vartheta\\
&\stackrel{\eqref{eq:gio}}{=}{16\pi^4}\sum_{\nu\in \Z^d\setminus\{0\}}\sum_{x,\,y\in \T^d_n}\|\nu\|^{-2\epsilon}n^{d-4}\left(n^d L^2+H(nx,\,ny)\right)\int_{B(x,\,\frac{1}{2n})}\mathbf e_{\nu}(\vartheta)\De\vartheta\int_{B(y,\,\frac{1}{2n})}\overline{\mathbf e_{\nu}(\vartheta)}\De\vartheta.
\end{align*}
Since $\int_{\T^d}\mathbf e_{\nu}(\vartheta)\De\vartheta=0$, the previous expression reduces to
\[
{16\pi^4}\sum_{\nu\in \Z^d\setminus\{0\}}\sum_{x,\,y\in \T^d_n}\|\nu\|^{-2\epsilon}n^{d-4}H(nx,\,ny)\int_{B(x,\,\frac{1}{2n})}\mathbf e_{\nu}(\vartheta)\De\vartheta\int_{B(y,\,\frac{1}{2n})}\overline{\mathbf e_{\nu}(\vartheta)}\De\vartheta.
\]
From this point onwards, the computations of the proof of (P1) can be repeated in a one-to-one fashion.
\subsection{Truncation method}
At the moment we are able to determine the scaling limit when the weights are bounded almost surely. To lift this condition to zero mean and finite variance only, we begin by defining a truncated field and show it will determine the scaling limit of the global field. Fix an arbitrarily large (but finite) constant $\cR>0$. Set
\begin{align*}
w_n^{<\cR}(x)&:=\frac{1}{2d}\sum_{y\in \Z_n^d}g(x,\,y)\sigma(y)\one_{\{|\sigma(y)|< \cR\}},\\
w_n^{\ge\cR}(x)&:=\frac{1}{2d}\sum_{y\in \Z_n^d}g(x,\,y)\sigma(y)\one_{\{|\sigma(y)|\ge \cR\}}.
\end{align*}
Clearly $w_n(\cdot)=w_n^{<\cR}(\cdot)+w_n^{\ge\cR}(\cdot)$. To prove our result, we will use
\begin{theorem}[{\citet[Theorem~4.2]{Bi68}}]\label{thm:billy}
Let $S$ be a metric space with metric $\rho$. Suppose that $(X_{n,\,u},\,X_n)$ are elements of $S\times S.$ If
\[
\lim_{u\to+\infty} \limsup_{n\to+\infty} \prob\left(\rho(X_{n,\,u},\,X_n)\ge \tau\right)=0
\]
for all $\tau>0$, and $X_{n,\,u}\Rightarrow_{n}Z_u\Rightarrow_u X$, where $``\Rightarrow_x''$ indicates convergence in law as $x\to +\infty$, then $X_n\Rightarrow_n X$.
\end{theorem}
Following this Theorem, we need to show two steps:
\begin{itemize}
\item[(S1)] $\lim_{\cR\to+\infty} \limsup_{n\to+\infty} \prob\left(\left\|\Xi_{w_n}-\Xi_{w_n^{<\cR}}\right\|_{\mathcal H_{-\eps}}\ge \tau\right)=0$ for all $\tau>0$.
\item[(S2)] For a constant $v_\cR>0$, we have $\Xi_{w_n^{<\cR}}\Rightarrow_n \sqrt{v_\cR}\, \Xi\Rightarrow_{\cR}\Xi$ in the topology of $\mathcal H_{-\eps}$.
\end{itemize}
As a consequence we will obtain that $\Xi_{w_n}$ converges to $\Xi$ in law in the topology of $\mathcal H_{-\eps}.$
\subsubsection{Proof of (S1)}
We notice that
\[
\left\| \Xi_{w_n}-\Xi_{w_n^{<\cR}}\right\|_{\mathcal H_{-\eps}}=\left\| \Xi_{w_n^{\ge\cR}}\right\|_{\mathcal H_{-\eps}}
\]
by definition, for every realization of $(\sigma(x))_{x\in \Z_n^d}$. Since, for every $\tau>0$,
\[
\prob\left(\left\| \Xi_{w_n^{\ge\cR}}\right\|_{\mathcal H_{-\eps}}\ge\tau\right)\le\frac{\E\left[\left\| \Xi_{w_n^{\ge\cR}}\right\|_{\mathcal H_{-\eps}}^2\right]}{\tau^2}
\]
it will suffice to show that the numerator on the right-hand side goes to zero to show (S1). But
\begin{align}
&\E\left[\left\| \Xi_{w_n^{\ge\cR}}\right\|_{\mathcal H_{-\eps}}^2\right]\nonumber\\
&\label{eq:withRandnot}=16\pi^4\sum_{\nu\in \Z^d\setminus\{0\}}\sum_{x,\,y\in \T^d_n}\|\nu\|^{-4\eps}n^{d-4}\E\left[ w_n^{\ge\cR}(xn) w_n^{\ge\cR}(ny)\right]\int_{B(x,\,\frac{1}{2n})}\mathbf e_{\nu}(\vartheta)\De\vartheta\int_{B(y,\,\frac{1}{2n})}\overline{\mathbf e_{\nu}(\vartheta)}\De\vartheta
\end{align}
Since the $\sigma$'s are i.i.d., we see that
\begin{align}
\E & \left[w_{n}^{\ge R}(xn)w_{n}^{\ge \cR}(yn)\right]=\frac{1}{4d^{2}}\sum_{w\in\mathbb{Z}_n^{d}}g(nx,\,w)g(ny,\,w)\E\left[\sigma(w)^{2}\mathbbm{1}_{\left\{ |\sigma(w)|\ge \cR\right\} }\right]\nonumber\\
 & +\frac{1}{4d^{2}}\sum_{w\neq v\in\mathbb{Z}_n^{d}}g(nx,\,w)g(ny,\,v)\E\left[\sigma(w)\sigma(v)\mathbbm{1}_{\left\{ |\sigma(w)|\ge \cR\right\} }\mathbbm{1}_{\left\{ |\sigma(v)|\ge \cR\right\} }\right]\nonumber\\
 & =\left(\E\left[\sigma^2\mathbbm{1}_{\left\{ |\sigma|\ge \cR\right\} }\right]-\E\left[\sigma\mathbbm{1}_{\left\{ |\sigma|\ge \cR\right\} }\right]^{2}\right)\frac{1}{4d^{2}}\sum_{w\in\mathbb{Z}_n^{d}}g(nx,\,w)g(ny,\,w)\nonumber\\
 & +\E\left[\sigma\mathbbm{1}_{\left\{ |\sigma|\ge \cR\right\} }\right]^2\frac{1}{4d^{2}}\sum_{w,\,v\in\mathbb{Z}^{d}_n}g(nx,\,w)g(ny,\,v).\label{eq:ex_toplugin}
\end{align}
Pluging the last expression into~\eqref{eq:withRandnot} gives two terms. The first one is, using \eqref{eq:gio}, equal to
\begin{align*}16\pi^4&\left(\E\left[\sigma^{2}\mathbbm{1}_{\left\{ |\sigma|\ge \cR\right\} }\right]-E\left[\sigma\mathbbm{1}_{\left\{ |\sigma|\ge \cR\right\} }\right]^{2}\right)\times\\
&\times\sum_{\nu\in \Z^d\setminus\{0\}}\|\nu\|^{-4\eps}n^{d-4}\sum_{x,\,y\in \T^d_n}H(nx,\,ny)F_{n,\,\nu}(x)\overline{F_{n,\,\nu}(y)}
\end{align*}
where $F_{n,\,\nu}(x)$ was defined as $\int_{B(x,\,\frac{1}{2n})}\mathbf e_{\nu}(\vartheta)\De\vartheta$. We have at hand \eqref{eq:first-aid}, which we can use to upper-bound the previous expression by
\begin{align*}
C'16\pi^4\left(\E\left[\sigma(w)^{2}\mathbbm{1}_{\left\{ |\sigma(w)|\ge \cR\right\} }\right]-E\left[\sigma(w)\mathbbm{1}_{\left\{ |\sigma(w)|\ge \cR\right\} }\right]^{2}\right)\sum_{\nu\in \Z^d\setminus\{0\}}\|\nu\|^{-4\eps}
\end{align*}
for some $C'>0$.
The sum over $\nu$ is finite as long as $\eps>{d}/{4}$, and $$\E\left[\sigma(w)^{2}\mathbbm{1}_{\left\{ |\sigma(w)|\ge \cR\right\} }\right]-E\left[\sigma(w)\mathbbm{1}_{\left\{ |\sigma(w)|\ge \cR\right\} }\right]^{2}$$ is going to zero as $\cR\to+\infty$ (note that $\sigma$ has finite variance). We will show that the second term obtained by inserting the second summand of~\eqref{eq:ex_toplugin} in~\eqref{eq:withRandnot} is zero to complete the proof of (S1). In fact we obtain
\begin{align*}
\frac{4\pi^4}{d^2}n^{d-4}&\E\left[\sigma\mathbbm{1}_{\left\{ |\sigma|\ge \cR\right\} }\right]^{2}\sum_{\nu\in \Z^d\setminus\{0\}}\|\nu\|^{-4\eps}\times\\
&\times\sum_{x,\,y\in \T^d_n}\sum_{w,\,v\in\mathbb{Z}_n^{d}}g(nx,\,w)g(ny,\,v) F_{n,\,\nu}(x)\overline{F_{n,\,\nu}(y)}.
\end{align*}
We consider the second line in the previous expression to deduce that it equals
\begin{align*}
&\left|\sum_{x\in \T^d_n}\sum_{w\in\mathbb{Z}_n^{d}}g(nx,\,w)F_{n,\,\nu}(x)\right|^2=n^{2d}\left|\sum_{w\in\mathbb{Z}_n^{d}}\sum_{x\in \Z^d_n}\widehat{g_w}(x)\widehat{F_{n,\,\nu}}(x)\right|^2\\
&\stackrel{\eqref{eq:20}}{=}\left|-2d\sum_{w\in\mathbb{Z}_n^{d}}\sum_{x\in \Z^d_n\setminus\{0\}}\frac{\psi_{-x}(w)}{\lambda_x}\widehat{F_{n,\,\nu}}(x)+\sum_{w\in\Z_n^d}\widehat{g_w}(0)\widehat{F_{n,\,\nu}}(0)\right|^2
\end{align*}
where Parseval's theorem was used in the first equality. Both the summands above are zero: the first because
\[
\sum_{w\in \Z_n^d}\psi_{-x}(w)=n^d\langle \psi_0,\,\psi_{-x}\rangle=0,\quad x\neq 0,
\]
the second because $\mathbf e_{\nu}$ has zero average and so
\[
\widehat{F_{n,\,\nu}}(0)=n^{-d}\sum_{y\in \Z_n^d}F_{n,
\,\nu}(y)=0.
\]
\subsubsection{Proof of (S2)}
Our idea is to use the computations we did for the case in which $\sigma$ is bounded a.~s. since we are imposing that $|\sigma|<\cR$. However we have to pay attention to the fact that $\sigma\one_{\{|\sigma|<\cR\}}$ is not centered anymore, but has mean $m_\cR:=\E[\sigma\one_{\{|\sigma|<\cR\}}]$, nor has variance $1$, but $v_{\cR}:=\var[\sigma\one_{\{|\sigma|<\cR\}}]$. However we can circumvent this by using our previous results. If we set
\[
\sigma^\cR(x):=\sigma(x)\one_{\{|\sigma(x)|<\cR\}}-m_\cR
\]
we can consider the field
\[
\Xi_{n,\,\cR}(x):=\frac{4\pi^2}{2d}{n^{\frac{d-4}{2}}}\sum_{z\in \T^d_n}\sum_{w\in\Z_n^d}g(w,\,nz)\sigma^\cR(w)\one_{B(z,\,\frac{1}{2n})}(x),\quad x\in \T^d.
\]
Since $(2d)^{-1}\sum_{y\in \Z_n^d}g(\cdot,\,y)$ is a constant function on $\Z_n^d$ it follows that
\[
\la \Xi_{n,\,\cR},\,u  \ra=\la \Xi_{w_n^{<\cR}},\,u\ra
\]
for all smooth functions $u$ with zero average. Hence the field $\Xi_{n,\,\cR}$ has the same law of $\Xi_{w_n^{<\cR}}$. If we multiply and divide the former by $\sqrt{v_\cR}$, we obtain
\[
\Xi_{n,\,\cR}=\sqrt{v_\cR}\frac{4\pi^2}{2d}{n^{\frac{d-4}{2}}}\sum_{z\in \T^d_n}\sum_{w\in\Z_n^d}g(w,\,nz)\frac{\sigma^\cR(w)}{\sqrt{v_\cR}}\one_{B(z,\,\frac{1}{2n})}(x),\quad x\in \T^d.
\]
Since now the weights ${\sigma^\cR(w)}(v_\cR)^{-\frac{1}{2}}$ satisfy the assumptions of Theorem~\ref{thm:3}, we know that the above field will converge to $\sqrt{v_\cR}\,\Xi$ in law. Using the covariance structure of the limiting field, the fact that the field is Gaussian, and $\lim_{\cR\to+\infty}\sqrt{v_\cR}= 1$, a straightforward computation shows that $\sqrt{v_\cR}\,\Xi$ converges in law to $\Xi$ in the topology of $\mathcal H_{-\eps}$. With Theorem~\ref{thm:billy} we can conclude.
\section{Proof of Theorem~\ref{corol:kernel}}\label{sec:corol}
\paragraph{\it Preliminaries.}We must conclude with the proof of Theorem~\ref{corol:kernel} and begin by introducing some notation. We take $\zeta$, an (arbitrary) smooth radial function on $\R^d$, such that
\eq{}\label{eq:zeta}
\begin{cases}
\zeta(x)=1 &\|x\|\ge \frac{1}{2},\\
\zeta(x)=0 & \|x\|\le \frac{1}{4}.
\end{cases}
\eeq{}
Let us call 
\[G(x):=\zeta(x)\|x\|^{-4}=\|x\|^{-4}+(\zeta(x)-1)\|x\|^{-4}\]
and let $\mathcal G_d$ be its Fourier transform (in the sense of distributions) 
\[
\mathcal G_d(x):= \widehat G(x).
\]
Since $(\zeta(\cdot)-1)\|\cdot\|^{-4}$ is a compactly supported distribution, its Fourier transform will be a smooth function which we call $h_d$. Using the results on $\widehat{\|\cdot\|^{-4}}$ contained in Example~2.4.9 of \cite{Grafakos}, we have the explicit description of $\mathcal G_d$ in \eqref{eq:gexp}. In particular $\mathcal G_d$ decays faster than the reciprocal of any polynomial function at infinity. To see this, recall that $\widehat{D^\alpha G}(x)=\left(2\pi\i x\right)^{|\alpha|}\mathcal G_d(x)$, for any multi-index $\alpha$. If the order of the derivative is large enough (precisely $|\alpha|>d-4$), then $D^\alpha G(x)\in L^1(\R^d)$; in this case, $\left(2\pi\i x\right)^{|\alpha|}\mathcal G_d(x)$ is bounded on $\R^d$ and hence $\left|\mathcal G_d(x)\right|\le C \|x\|^{-N}$ for every positive integer $N$ as $\|x\|\to+\infty$. Let us denote by $f_\kappa:=\mathcal G_d\ast \phi_\kappa$ and note that
\eq{}\label{eq:bee}\widehat{f_\kappa}(\cdot)=\widehat{\mathcal G_d}(\cdot)\widehat{\phi_\kappa}(\cdot)=\zeta(\cdot)\|\cdot\|^{-4}\widehat{\phi_\kappa}(\cdot).\eeq{}
It follows that for some $C>0$ (depending on $\kappa$), 
\eq{}\label{eq:prop_g_eps}
\left|\widehat{f_\kappa}(\cdot)\right| \le C(1+\|\cdot\|)^{-d-1}.
\eeq{}
Moreover 
\eq{}\label{eq:decay_g_only}
\left|f_\kappa(\cdot)\right|\le C\left(1+\|\cdot\|\right)^{-d-1}\eeq{}
near infinity thanks to the rapid decay of $\mathcal G_d$ at infinity; furthermore $\mathcal G_d$ is integrable near zero in $d\ge 5$ by \eqref{eq:gexp}. Hence $f_\kappa$ is $C^\infty(\R^d)$ and also in $L^1(\R^d)$.
Using $f_\kappa= \mathcal G_d\ast\phi_\kappa$ and the definition of $\zeta$ we have that
\begin{align}\eqref{eq:fish} &=n^{-2d}\sum_{z,\,z'\in\T_n^d} u(z) u(z') \sum_{w\in\Z_n^d}\widehat{\phi_\kappa}(w)\zeta(w)\frac{\exp(2\pi \i(z-z')\cdot w)}{\|w\|^4}\nonumber\\
&=n^{-2d}\sum_{z,\,z'\in\T_n^d} u(z) u(z') \sum_{w\in\Z_n^d}\widehat{f_\kappa}(w){\exp(2\pi \i(z-z')\cdot w)}.\label{eq:fish_2}\end{align}
Now we can rewrite this term as
\begin{align}
&n^{-2d}\sum_{z,\,z'\in\T_n^d} u(z) u(z') \sum_{w\in\Z^d}\widehat{f_\kappa}(w){\exp(2\pi \i(z-z')\cdot w)}\nonumber\\
&-n^{-2d}\sum_{z,\,z'\in\T_n^d} u(z) u(z') \sum_{w\in \Z^d:\,\|w\|_\infty>n}\widehat{f_\kappa}(w){\exp(2\pi \i(z-z')\cdot w)}.\label{eq:gif}
\end{align}
First we show the second term above is negligible in the following Lemma.
\begin{lemma}\label{claim:claim1}
\begin{equation*}
\lim_{n\to+\infty} n^{-2d}\sum_{z,\,z'\in\T_n^d} u(z) u(z') \sum_{w\in \Z^d:\,\|w\|_\infty>n}\widehat{f_\kappa}(w)\exp(2\pi \i(z-z')\cdot w)=0.
\end{equation*}
\end{lemma}
\begin{proof} Note that
\begin{align*}
&n^{-2d}\left|\sum_{z,\,z'\in\T_n^d} u(z) u(z') \sum_{w\in \Z^d:\,\|w\|_\infty>n}\widehat{f_\kappa}(w)\exp(2\pi \i(z-z')\cdot w)\right|\\
&=\left|\sum_{w\in \Z^d:\,\|w\|_\infty>n}\widehat{f_\kappa}(w) \left(n^{-d}\sum_{z\in\T_n^d} u(z)\exp(2\pi \i z\cdot w)\right)\left(n^{-d}\sum_{z'\in\T_n^d} u(z')\exp(-2\pi \i z'\cdot w)\right)\right|\\
&\le \|u\|_{L^\infty(\T^d)}^2 \sum_{w\in \Z^d:\,\|w\|_\infty>n} \left|\widehat{f_\kappa}(w)\right|\le C \|u\|_{L^\infty(\T^d)}^2 \sum_{w\in \Z^d:\,\|w\|_\infty>n}\frac1{(1+\|w\|)^{d+1}}\le C \|u\|_{L^\infty(\T^d)}^2 n^{-1}
\end{align*}
thanks to \eqref{eq:prop_g_eps} and the Euler-MacLaurin formula \cite[Theorem~1]{Apostol}. This shows Lemma~\ref{claim:claim1}.
\end{proof}
Therefore, rather than working on \eqref{eq:fish_2}, we will concentrate on the first term of \eqref{eq:gif}.
\begin{proof}[Proof of Theorem~\ref{corol:kernel}]
Following the proof of Proposition~\ref{lem:bhindi}, it is enough to prove the convergence of the first term of \eqref{eq:gif} to the right-hand side of \eqref{eq:kernel_exp}. Since $f_\kappa$ and $\widehat{f_\kappa}$ satisfy the assumptions of the Poisson summation formula \cite[Corollary~2.6, Chapter VII]{stein:weiss}, we apply it to \eqref{eq:fish_2} and obtain 
\begin{align}
&\lim_{n\to+\infty}n^{-2d}\sum_{z,\,z'\in\T_n^d} u(z) u(z') \sum_{w\in \Z^d}\widehat{f_\kappa}(w)\exp(2\pi \i(z-z')\cdot w)\nonumber\\
&=\lim_{n\to+\infty}n^{-2d}\sum_{z,\,z'\in\T_n^d} u(z) u(z') \sum_{w\in \Z^d} f_\kappa((z-z')+w)\nonumber\\
&= \lim_{n\to+\infty}\sum_{w\in \Z^d} n^{-2d}\sum_{z,\, z'\in \T_n^d}u(z) u(z')f_\kappa((z-z')+w).\label{eq:switch}
\end{align}
We would then like to exchange sum and limit and thus we shall justify the use of the dominated convergence theorem. To this purpose we need to observe that $\|z-z'\|\le \sqrt d$ so that $\left|\|z-z'+w\|-\|w\|\right|\le 2\sqrt{d}.$ Therefore
\begin{align}
& \sum_{w\in \Z^d} n^{-2d}\sum_{z,\, z'\in \T_n^d}\left| u(z) u(z')f_\kappa((z-z')+w)\right|\nonumber\\
 &\stackrel{\eqref{eq:decay_g_only}}{\le} C n^{-2d}\|u\|_{L^\infty(\T^d)}^2\sum_{w\in \Z^d:\,\|w\|_\infty>\sqrt{d}}\,\sum_{z,\, z'\in \T_n^d}\frac{1}{(1+\|z-z'+w\|)^{d+1}}\nonumber\\
 &+C n^{-2d}\|u\|_{L^\infty(\T^d)}^2\sum_{w\in \Z^d:\,\|w\|_\infty\le \sqrt{d}}\,\sum_{z,\, z'\in \T_n^d}\frac{1}{(1+\|z-z'+w\|)^{d+1}}.\label{eq:two}
\end{align}
The second term can be directly bounded by a constant independent of $n$, being a finite sum.
As for the first term in \eqref{eq:two} we have by the Euler-MacLaurin formula 
\begin{align}
Cn^{-2d}&\|u\|_{L^\infty(\T^d)}^2\sum_{w\in \Z^d:\,\|w\|_\infty>\sqrt{d}}\,\sum_{z,\, z'\in \T_n^d}\frac{1}{(1+\|z-z'+w\|)^{d+1}}\nonumber\\
&\le C n^{-2d}\|u\|_{L^\infty(\T^d)}^2\sum_{w\in \Z^d:\,\|w\|_\infty>\sqrt{d}}\,\sum_{z,\, z'\in \T_n^d}\frac{1}{(1-2\sqrt{d}+\|w\|)^{d+1}}\nonumber\\
&\le C\left(\int_{\sqrt d-1}^{+\infty}\frac{\rho^{d-1}}{(1-2\sqrt{d}+\rho)^{d+1}}\De \rho+c\right)\le c \label{eq:three}
\end{align}
where $C,\,c$ are independent of $n$ in each occurence above. These inequalities plugged into \eqref{eq:two} give the desired bound which allows us to switch summation and limit in \eqref{eq:switch}. Going on and using also the smothness of $f_\kappa$ we compute 
\begin{align*}
&\lim_{n\to+\infty}\sum_{w\in \Z^d} n^{-2d}\sum_{z,\, z'\in \T_n^d}u(z) u(z')f_\kappa((z-z')+w)\\
&=\sum_{w\in \Z^d} \iint_{\T^d\times \T^d} u(z)u(z')f_\kappa((z-z')+w)\De z\De z'.
\end{align*}
The fast decay of $\mathcal G_d$ and hence of $f_\kappa$ at infinity enables us to apply the dominated convergence again to finally arrive at
\begin{align*}
\lim_{\kappa\to 0}\sum_{w\in \Z^d}& \iint_{\T^d\times \T^d} u(z)u(z')f_\kappa((z-z')+w)\De z\De z'\\
&=\sum_{w\in \Z^d} \iint_{\T^d\times \T^d} u(z)u(z')\mathcal G_d((z-z')+w)\De z\De z'.
\end{align*}
Due to polynomial decay of $\mathcal G_d$ at infinity it is immediate to exchange sum and integrals to derive \eqref{eq:kernel_exp}.
\end{proof}
\bibliographystyle{abbrvnat}
\bibliography{literaturASP}
\end{document}